\documentclass[11pt, reqno]{amsart}
\usepackage{amssymb,pb-diagram}
\usepackage{amsrefs}
\usepackage{fancyhdr}

\usepackage{amsmath}

\newtheorem{theorem}{Theorem}[section]
\newtheorem{proposition}[theorem]{Proposition}
\newtheorem{corollary}[theorem]{Corollary}
\newtheorem{lemma}[theorem]{Lemma}

\newtheorem{example}[theorem]{Example}

\theoremstyle{definition}

\theoremstyle{remark}
\newtheorem{remark}[theorem]{Remark}

\theoremstyle{definition}

\def\cdots{\mathinner{\cdotp\cdotp\cdotp}}

\def \diam{\text{diam }}

\newcommand{\ds}{\displaystyle}

\newcommand{\A}{\mathcal A} 
 \newcommand{\D}{\mathcal D}

 \newcommand{\M}{M}

\newcommand{\R}{\ensuremath{\mathbb{R}}}
\newcommand{\MM}{\ensuremath{\mathcal{M}}}
\newcommand{\NN}{\ensuremath{\mathcal{N}}}
\newcommand{\T}{\ensuremath{\mathcal{T}}}
\newcommand{\Mb}{\ensuremath{\mathbb{M}}}
\newcommand{\N}{\ensuremath{\mathbb{N}}}

\newcommand{\Z}{\ensuremath{\mathbb{Z}}}

\newcommand{\coarse}{\ensuremath{\underset{coarse}{\lhook\joinrel\relbar\joinrel\rightarrow}}}

\newcommand{\strong}{\ensuremath{\underset{strong}{\lhook\joinrel\relbar\joinrel\rightarrow}}}

\newcommand{\isometricword}{\ensuremath{\underset{isometric}{\lhook\joinrel\relbar\joinrel\relbar\joinrel\relbar\joinrel\rightarrow}}}

\newcommand{\lip}{\ensuremath{\underset{Lip}{\lhook\joinrel\relbar\joinrel\rightarrow}}}

\def\diam{\text{diam }}

\def\supp{\text{supp}}

\newcommand{\mk}{\medskip}

\allowdisplaybreaks

\begin{document}

\title[Quantitative strong embeddings]{Quantitative nonlinear embeddings into Lebesgue sequence spaces}

\author[F. Baudier]{Florent Baudier}
\address{Mathematics Department\\ Texas A\&M University\\College Station, TX 77843 USA}
\email{florent@math.tamu.edu}



\date{}

\dedicatory{}

\begin{abstract}In this paper fundamental nonlinear geometries of Lebesgue sequence spaces are studied in their quantitative aspects. Applications of this work are a positive solution to the strong embeddability problem from $\ell_q$ into $\ell_p$ ($0<p<q\le 1$) and  new insights on the coarse embedabbility problem from $L_q$ into $\ell_q$, $q>2$. Relevant to geometric group theory purposes, the exact $\ell_p$-compressions of $\ell_2$ are computed. Finally coarse deformation of metric spaces with property A and locally compact amenable groups is investigated.
\end{abstract}

\maketitle


\section{Introduction}

\subsection{Motivation and organization of the article}\ \\
We start by explaining our motivation to write this paper and describe quite in details the actual panorama of the coarse, uniform and strong geometries of the classical Lebesgue spaces. We invite the reader to look up in Section \ref{notation} for the definitions or technical terms used in this introductory section.

\mk

The original motivation was to study nonlinear embeddings from $\ell_p$ into $\ell_q$ when $0<q<p\le 1$. In \cite{Albiac2008} it was proven that there is no nontrivial Lipschitz map, hence no bi-Lipschitz embedding, from $\ell_p$ into $\ell_q$ when $0<q<p\le 1$. The question was raised in \cite{AlbiacBaudier2013} wether or not there exists a weaker type of embedding. In the present paper the existence of a strong embedding is proven and quantitative estimates for the compression and expansion moduli are given. It turns out that the technique used in this paper to construct this quantitative strong embedding is directly inspired by techniques introduced by geometric group theorists and is naturally relevant to the \textit{quantitative} study of embeddings into sequence spaces of metrics of Lipschitz-negative type, non-discrete metric spaces with property A and locally compact amenable groups. It is also clear that estimating the compression exponent of a coarse embedding between two Banach spaces is a fundamental question in geometric group theory. The reader interested in the geometric group theoretic applications would certainly find enlightening to consult the monograph from Nowak and Yu \cite{NowakYu}.  

\smallskip

The theory of Lipschitz (or coarse) embeddings of locally finite metric spaces into $L_p$ or $\ell_p$ is essentially the same. Indeed, it is known that Lipschitz (or coarse) embeddability into Banach spaces is finitely determined for locally finite metric spaces (see \cite{Ostrovskii2012}). Therefore to embed a locally finite metric space into $\ell_1$, let's say, it is sufficient to embed into $L_1$ since every finite subset of $L_1$ is isometric to  subset of $\ell_1$. Note that the quality of a coarse embedding would be preserved up to some multiplicative constants. In particular the $\ell_p$-compression of a finitely generated group coincides with its $L_p$-compression. However the situation is drastically different for non-locally finite metric spaces. For instance $\ell_2$ is linearly isometric to a subspace of $L_p$ $(1\le p\neq 2<\infty)$ but does not bi-Lipschitzly embed into $\ell_p$. In general, it seems very unlikely that one can derive a coarse embedding into $\ell_p$ out of a coarse embedding into $L_p$ maintaining the same quality since it follows from the work of Kalton and Randrianarivony \cite{KaltonRandrianarivony2008} (see also \cite{Lancien}) that $L_p$ does not admit an embedding into $\ell_p$ which is bi-Lipschitz for large distances (a.k.a quasi-isometric embedding in geometry group theory dictionary). One of the objectives of this article is to provide quantitative strong embeddings of certain non-locally finite metric spaces into ``small" $L_p(\mu)$-spaces. For instance, the target spaces for all the embeddings that are to be constructed are of the form $\ell_p(\Gamma)$ for some set $\Gamma$.

\smallskip

Coarse embeddings of groups into Hilbert spaces and its quantitative relative, namely the Hilbert compression, have numerous applications in Noncommutative Geometry and Topology. However, Hilbert spaces are, in some sense, the most difficult spaces to embed into in the coarse category. First of all a separable Hilbert space can be coarsely embedded in every Banach space with an unconditional basis and finite cotype \cite{Ostrovskii2009}. This rather large class contains the classical separable $L_p$-spaces. On the other hand Naor and Peres proved that for finitely generated groups that the equivariant $L_2$-compression is actually minimal among all the other equivariant $L_p$-compressions $(p\ge 1)$ (Lemma $2.3$ in \cite{NaorPeres2008}). We want to point out that it follows from \cite{Ostrovskii2009}, or \cite{Baudier2012}, and Dvoretzky's theorem that for locally finite metric spaces the non-equivariant Hilbert-compression is minimal among all the other non-equivariant $Y$-compressions for any infinite dimensional Banach space $Y$. It does not seem that this fact has been noticed before. Embeddings involving a separable Hilbert space either as the domain space or the target space are studied from the quantitative standpoint in this paper. The present article is divided into two parts and is organized as follows: 

\smallskip

In the first part, quantitative nonlinear embeddings between Lebesgue sequence spaces are investigated. A fine quantitative analysis of a generalization of a device, introduced by Guentner and Dadarlat, to construct coarse embeddings into Hilbert spaces is detailed. This device is used to produce strong deformation gaps of separable Hilbert spaces into $\ell_q$-sums of metric linear spaces, in particular into the classical Lebesgue sequence spaces. A natural application to the embedding of metrics of Lipschitz-negative type is given. Finally the exact $\ell_q$-compression of $\ell_p$ is computed using a very powerful result of Kalton and Randrianarivony. As an application of this computation we give derive some very interesting new information regarding the important open problem of the coarse embeddability of $L_p$ into $\ell_p$ when $p>2$.

\smallskip

In the second part, using essentially metric techniques, rather general coarse deformation gaps when embedding non-discrete metric measured spaces with property A and locally compact amenable groups are given. The behavior of the compression modulus is related to a radial dilation parameter for sequences of sets associated to property A or amenability.

\subsection{Notation and terminology}\label{notation}\ \\

Throughout this paper one will use the convenient notation $A\lesssim B$ (resp. $A\gtrsim B$) if there is a universal constant $C$ such  that $A\le CB$ (resp. $A\ge CB$), A and B being real numbers or functions. 

\smallskip

\noindent A metric space will be said to be \textit{uniformly discrete} if there exist $\varepsilon>0$ such that $d(x,y)\ge \varepsilon$ for all pair of distinct points.  

\smallskip

\noindent Recall that for a non-decreasing function $T:\R\to\R$ the \textit{generalized inverse} is the function  $T^-:\R\to[-\infty,\infty]$ defined by $$T^-(y)=\inf\{x\in \R,T(x)\ge y\}$$ with the convention that $\inf\emptyset=\infty$. If $T$ is increasing and continuous $T^-$ coincides with $T^{-1}$, the ordinary inverse of $T$, on the range of $T$.

\medskip

\paragraph{\textbf{Countable $\ell_p$-sum of pointed metric spaces for $0<p<\infty$}}\ \\
Let $(\MM_n,\delta_n)_{n\ge 1}$ be a sequence of metric spaces. Pick a point $O_n$ in each space. The point $O_n$ is always chosen to be $0$ when $\MM_n$ is a metric linear space.
The $\ell_p$-sum of the sequence of pointed metric spaces $(\MM_n)_{n\ge 1}$ is the space $$\ds\left(\sum_{n=1}^\infty \MM_n\right)_{p}:=\left\{z=(z_n)_{n\ge1}\in\prod_{n=1}^{\infty} \MM_n;\  \vert z\vert_p:=\sum_{n=1}^{\infty}\delta_n(z_n,O_n)^p<\infty\right\},$$ equipped with the distance $$\Delta_p(x,y):=\sum_{n=1}^{\infty} \delta_n(x_n,y_n)^p,\ 0<p\le 1$$
or $$\Delta_p(x,y):=\left(\sum_{n=1}^{\infty} \delta_n(x_n,y_n)^p\right)^{1/p},\ p\ge 1$$
If $\MM_n=\MM$ and $O_n=t_0\in\MM$ for every $n$ one uses the simpler and classical notation $\ell_p(\N,\MM;t_0)$ or simply $\ell_p(\MM)$. It is clear that if the spaces $\MM_n$ are metric linear spaces then the $\ell_p$-sum is also a linear metric space for the canonical operations. It is also clear that if the metrics are translation invariant with respect to some group structure the spaces may carry then the metrics $\Delta_p$ are also translation invariant for the canonical group structure.
Basic examples of $\ell_p$-sum of metric spaces are the classical Lebesgue sequence spaces $\ell_p$ of $p$-summable real sequences. Recall that for $p\ge 1$, its natural distance is induced by the classical $\ell_p$-norm $\Vert x\Vert_{p}=\left(\ds\sum_{n=1}^\infty \vert x_n\vert^p\right)^{1/p}.$ For values of $p$ on the other side of the spectrum, $0<p\le 1$ its standard distance is given by 
 $$d_{p}(x,y)=\sum_{n=1}^\infty \vert x_n-y_n\vert^p.$$
 
In both cases $\ell_p$ can be seen as the $\ell_p$-sum of countably many copies of the metric space $(\R,\vert\cdot\vert)$, endowed with the ad-hoc metric $\Delta_p$, i.e. $\ell_p(\N,\R)$. In the sequel $\ell_p$ for $0<p<\infty$ shall always be considered as a metric space with the metric $d_{p}(x,y)$ if $0<p\le 1$ or with the metric $\Vert x-y\Vert_{p}$ if $p\ge 1$. Note that for $p=1$ both metrics coincide and there is no confusion possible.

\medskip

\paragraph{\textbf{Compression and expansion moduli}}\ \\
Let $(\D,d_\D)$ and $(\T,d_\T)$ be two metric spaces and $f:\D\to \T$. One defines 
$$\rho_f(t)=\inf\{d_\T(f(x),f(y)) : d_\D(x,y)\geq t\},$$
and 
$$\omega_f(t)={\rm sup}\{d_\T(f(x),f(y)) : d_\D(x,y)\leq t\}.$$

\noindent Remark that for every $x,y\in \D$, $$\rho_f(d_\D(x,y))\le d_\T(f(x),f(y))\le\omega_f(d_\D(x,y)).$$ The moduli $\rho_f$ and $\omega_f$ will  be respectively called the \textit{compression modulus} and the \textit{expansion modulus} of the embedding. 

\smallskip

When the metric spaces are not uniformly discrete the embedding is said to be a \textit{uniform embedding} if $\omega_f(t)\to 0$ when $t\to 0$ and $\rho_f(t)>0$ for all $t>0$.

\smallskip

If the domain and the target spaces are two unbounded metric spaces then $f:\D\to \T$ is a \textit{coarse embedding} if  $\rho_f(t)\to \infty$ when $t\to \infty$ and $\omega_f(t)<\infty$ for all $t>0$. 

\smallskip

A \textit{strong embedding} is an embedding which is \textit{simultaneously} coarse and uniform.

\smallskip

For coarsely continuous maps (in particular coarse embeddings) on metrically convex spaces, the expansion modulus is subadditive. In this case for every $t>s>0$, $\omega_f(t)\le \frac{2\omega_f(s)}{s}t$ and the expansion modulus is at most linear for large distances, i.e. for some $a,\tau>0$, $\omega_f(t)\le at$ for all $t\ge \tau$. The expansion modulus of a coarse embedding can always be assumed to be at most Lipschitz if the domain space is a uniformly discrete graph or a metrically convex space with the metric $d_\D^{\ge \tau}(x,y)=d_\D(x,y)+\tau$, $x\neq y$.
Regarding uniform embeddings the situation is not so clear. However the compression modulus of a uniform embedding into a Hilbert space can be taken to be Lipschitz for small distances if the domain space is an Abelian (see \cite{BenyaminiLindenstrauss2000} p. 192).
\subsection{Deformation Gaps}
A variant of the terminology from \cite{GoulnaraDrutuSapir2009} shall be used to describe the quality of an embedding. A $[\rho,\omega]$-embedding $f$ from $\D$ into $\T$ is an embedding such that $$\rho(d_\D(x,y))\le d_\T(f(x),f(y))\le\omega(d_\D(x,y)).$$
 A \textit{coarse deformation gap} of $\D$ in $\T$ is given by a pair of functions $[\rho,\omega]$ such that there exists a $[\rho,\omega]$-coarse embedding, i.e. a $[\rho,\omega]$-embedding satisfying $\rho(t)\to \infty$ when $t\to \infty$ and $\omega(t)<\infty$, $\forall t>0$. 
One shall use the convenient notation $(t^\alpha,t^\beta)$ if there is a coarse (resp. uniform) compression gap $[t^r,t^s]$ for every $r<\alpha$ (resp. $r>\alpha$) and every $s>\beta$ (resp. $s<\beta$). The quotient $\omega/\rho$ is called the the \textit{deformation ratio}. Obviously two embeddings can have the same deformation ratio but different deformation gaps. 

\medskip 
 
For two functions $g,h\colon \R\to\R$ we write $g\ll h$ if there exist $a,b,c>0$ such that $g(t)\le ah(bt)+c$ for every $t\in\R$. If $g\ll h$ and $h\ll g$ then we write $g\asymp h$. It is an equivalence relation and a function shall be identified with its equivalence class in the sequel. Following \cite{GoulnaraDrutuSapir2009} one says that $[\rho_1,\rho_2]$ is a \textit{compression gap} of $\D$ in $\T$ if there exists a $[\rho_1,t]$-coarse embedding of $\D$ into $\T$ and for every $[\rho,t]$-coarse embedding of $\D$ into $\T$ one has $\rho\ll\rho_2$. If $\rho_1\asymp\rho_2$, $\rho_1$ is called the \textit{compression function} of $\D$ in $\T$.

\medskip

In particular if $\alpha_\T(\D)$ is the supremum of all numbers $\alpha$ such that $t^\alpha\ll \rho$ and $\rho$ is the compression function of $\D$ in $\T$ then $\alpha_\T(\D)$ is the \textit{compression exponent} of $\D$ in $\T$ ($\T$-compression of $\D$ in short) introduced by Guentner and Kaminker \cite{GuentnerKaminker2004}. In other words, $\alpha_\T(\D)$ is the supremum of all numbers $0\le \alpha\le 1$ so that if $d_\D(x,y)$ is large enough, $$d_\D(x,y)^{\alpha}\lesssim d_\T(f(x),f(y))\lesssim d_\D(x,y).$$

Analogous notions for uniform or strong embeddings can be defined in a similar fashion, e.g. \textit{uniform deformation gaps}, \textit{strong deformation gaps}, etc... For instance $[\rho,\omega]$ will be a strong deformation gap of $\D$ in $\T$ if there exists a $[\rho,\omega]$-embedding with $\displaystyle\lim_{t\to \infty}\rho(t)=\infty$, $\ds\lim_{t\to 0}\omega(t)=0$, $\rho(t)>0$ and $\omega(t)<\infty$ for all $t>0$. By changing accordingly the equivalence relation it is possible to define \textit{expansion gaps}, \textit{expansion functions} and \textit{expansion exponents}. More precisely, the $\T$-expansion of $\D$, denoted $\beta_\T(\D)$ is the supremum of all numbers $0<\beta\le 1$ such that $$d_\D(x,y)\lesssim d_\T(f(x),f(y))\lesssim d_\D(x,y)^{\beta}$$ as long as $d_\D(x,y)$ is small enough.

\medskip 

To illustrate this new terminology, theorem 2.1 from \cite{Kalton2007} says that if $\MM$ is a stable metric space then it is possible to construct a reflexive space $R$ and a strong embedding from $\MM$ into $R$ with coarse compression gap $(t,t]$ and uniform deformation gap $[t,t)$. 
Another example is theorem 2.1 in \cite{Baudier2012}. Every proper metric space embeds with strong deformation gap $[\frac{t}{\log(t)^2},t]$ into any Banach space without finite cotype.

\smallskip

In geometric group theory it is customary to call a quasi-isometric embedding an embedding $f$ which is bi-large-scale Lipschitz, i.e. such that for all $x,y\in \D$ the inequalities 
$$\frac{1}{A}d_\D(x,y)-B\le d_\T(f(x),f(y))\le Ad_\D(x,y)+B$$ hold for some constants $A\ge1$ and $B\ge 0$. 

\medskip

It is clear that if $\D$ is uniformly discrete for some scale $\epsilon>0$ then the embedding is bi-Lipschitz for large distances, i.e. there exists $\tau>0$ and $A_\tau>0$ such that \begin{equation}\label{LLD}\frac{1}{A_\tau}d_\D(x,y)\le d_\T(f(x),f(y))\le A_\tau d_\D(x,y)
\end{equation}
whenever $d_\D(x,y)\ge\tau$. Any $\tau>\max\{AB,\epsilon\}$ and $A_\tau\ge\max\left\{\frac{\epsilon A+B}{\epsilon},\frac{\tau A}{\tau-AB}\right\}$ will do the job.

If the domain space is a Banach space $X$ then a bi-Lipschitz embedding for large distances is also a quasi-isometric embedding. Assume that equality (\ref{LLD}) holds and let $\Vert x-y\Vert_X<\tau$. Pick $z\in X$ such that $\Vert y-z\Vert_X=\tau$, $\tau\le\Vert x-z\Vert_X\le2\tau$ and $\Vert x-z\Vert_X=\Vert x-y\Vert_X+\Vert y-z\Vert_X$. By the triangle inequality $d_\T(f(x),f(y))\le A_\tau\Vert x-y\Vert_X+\tau(A_\tau+1)$. 
On the other hand, 
\begin{align*}
d_\T(f(x),f(y))& \ge d_\T(f(x),f(z))-d_\T(f(y),f(z))\\
 & \ge \frac{1}{A_\tau}\Vert x-y\Vert_X-A_\tau\Vert y-z\Vert_X\\
 & \ge \frac{1}{A_\tau}\Vert x-y\Vert_X-\tau A_\tau\\
\end{align*}
When the domain space is a Banach space an embedding is quasi-isometric if and only if it is bi-Lipschitz for large distances. This terminology shall be preferred in this paper since in nonlinear Banach space theory a quasi-isometry means a bi-Lipschitz embedding with distortion $1+\epsilon$ for every $\epsilon>0$. The terminology \textit{coarse-Lipschitz embedding} could be used as well (cf \cite{Kalton2008}).

\smallskip

It is also clear that if for some scale $\tau>0$ there is a bi-Lipschitz embedding for large distance $f$ from $X$ into $\T$, the mapping $g(x)=f(\frac{\tau}{\tau'}x)$ is a bi-Lipschitz embedding for large distance at scale $\tau'$. In addition if the target space is a Banach space $Y$ the expansion modulus can be taken to be $1$-Lipschitz after an appropriate rescaling. For two Banach spaces $X$ and $Y$ it follows from the preceeding discussion that the $Y$-compression of $X$  is the supremum of all numbers $0\le \alpha\le 1$ over all embeddings such that $$\Vert x-y\Vert_{X}^{\alpha}\lesssim \Vert f(x)-f(y)\Vert_Y \le \Vert x-y\Vert_X\textrm{ whenever } \Vert x-y\Vert\ge 1.$$ 

Another convenient notation shall be used repeatedly in this paper. For a $[\rho,\omega]$-embedding $f$ from $\D$ into $\T$, $$\rho(d_\D(x,y))\lesssim_l d_\T(f(x),f(y))\lesssim\omega(d_\D(x,y))$$ shall mean that the inequality $d_\T(f(x),f(y))\lesssim\omega(d_\D(x,y))$ holds for every $x,y\in \D$ and $\rho(d_\D(x,y))\lesssim d_\T(f(x),f(y))$ is satisfied only if $d_\D(x,y)\ge \tau$ for some scale $\tau>0$.\\Similarly, the notation $\lesssim_s$ shall be understood for an inequality satisfied for small distances only.

\section{Quantitative nonlinear embeddings into Lebesgue spaces}
The purpose of this section is to study strong and coarse deformation gaps of nonlinear embeddings into the classical Lebesgue sequence spaces. 

\subsection{Strong embeddings into $\ell_q$-sums of metric linear spaces}\ \\

\smallskip

In \cite{Kraus} Kraus remarked that the Dadarlat-Guentner criterion from \cite{DadarlatGuentner2003} can be used to produce strong embeddings when the domain space is a Hilbert space. Lemma \ref{stronggluing} is a quantitative version of this criterion suited for strong embeddings. It is written in greater generality since it does not add too much technicality and can take care of $\ell_p$-target spaces in the zone $0<p<1$. It will be the main tool for building strong embeddings into $\ell_q$-sums of metric linear spaces and will provide lower estimates for the parameter $\alpha_{\ell_q}(\ell_2)$. Already known estimates for $q>2$ are retrieved and the first such estimates in the case $0<q<2$ are obtained. Rather surprisingly the estimates for the range $1\le q<2$  are optimal. In addition the coarse embedding that can be derived from Dadarlat-Guentner criterion will still be presented at the end of this section. Despite it gives weaker estimates on the $\ell_q$-compression for separable Hilbert spaces, it seems to be a more flexible embedding and can be used in broader contexts. An occurrence of this is shown in Section 3. 

\begin{lemma}\label{stronggluing} Let $(\MM,d)$ be a metric space. Suppose that for every $n\ge 1$, $(X_n,\delta_n)$ is a translation invariant metric linear space and suppose further that there exist maps $\varphi_n\colon \MM\to X_n$ with compression modulus $\rho_n$ and expansion modulus $\omega_n$ satisfying 
\begin{enumerate}
\item $\omega_n(t)\le \epsilon_n \gamma(t)<\infty, \forall t>0,$ for some $q$-summable sequence $(\epsilon_n)_{n\ge 1}$ and function $\gamma$\\
\item $\delta_n(\varphi_n(x),\varphi_n(y))\ge \eta$ whenever $d(x,y)\ge s_n$, for some $\eta>0$ and nondecreasing unbounded sequence $(s_n)_{n\ge1}$\\
\item $\rho_n(t)\ge \mu_n \xi(t)$ whenever $t$ is small enough, for some $(\mu_n)_{n\ge 1}$ $q$-summable and function $\xi$.
\end{enumerate}

\medskip

Define $\phi\colon\MM\to \left(\ds\sum_{n=1}^\infty X_n\right)_{q}$ by $\phi(x)=(\varphi_n(x)-\varphi_n(t_0))_{n\ge 1}$. The following inequalities hold:

\medskip

\begin{itemize}
\item if $q\ge 1$
\begin{equation}
\Delta_q(\phi(x),\phi(y))\lesssim \gamma(d(x,y)).
\end{equation} 

\begin{equation}
\Delta_p(\phi(x),\phi(y))\gtrsim k^{1/q}  \textrm{ whenever }s_k\le d(x,y)\le s_{k+1}
\end{equation}

\begin{equation}
\Delta_q(\phi(x),\phi(y))\gtrsim \xi(d(x,y)) \textrm{ whenever } d(x,y) \textrm{ is small enough}.
\end{equation}

\bigskip

\item if $0<q\le 1$ 
\begin{equation}
\Delta_q(\phi(x),\phi(y))\lesssim \gamma(d(x,y))^{q},
\end{equation} 

\begin{equation}
\Delta_q(\phi(x),\phi(y))\gtrsim k \textrm{ whenever }s_k\le d(x,y)\le s_{k+1}
\end{equation}

\begin{equation}
\Delta_q(\phi(x),\phi(y))\gtrsim \xi(d(x,y))^{q} \textrm{ whenever } d(x,y) \textrm{ is small enough}.
\end{equation}
\end{itemize}
\end{lemma}

\begin{proof}
First of all, $\phi$ is well defined, indeed

\begin{align*}
\vert \phi(x)\vert_q &=\sum_{n=1}^{\infty}\delta_n(\varphi_n(x)-\varphi_n(t_0),0)^q\\
&=\sum_{n=1}^{\infty}\delta_n(\varphi_n(x),\varphi_n(t_0))^q\\
&\le \sum_{n=1}^{\infty}\epsilon_n^q\gamma(d(x,t_0))^q\\
&\le \gamma(d(x,t_0))^q\sum_{n=1}^{\infty}\epsilon_n^q<\infty
\end{align*}

For $0<q<\infty$,

\begin{align*}
\sum_{n=1}^{\infty}\delta_n(\varphi_n(x)-\varphi_n(t_0),\varphi_n(y)-\varphi_n(t_0))^q&=\sum_{n=1}^{\infty}\delta_n(\varphi_n(x),\varphi_n(y))^q\\ 
&\le \sum_{n=1}^{\infty}\epsilon_n^q\gamma(d(x,y))^{p}\\
&\le \gamma(d(x,y))^{q}\sum_{n=1}^{\infty}\epsilon_n^q
\end{align*}

Now if $s_{k}\le d(x,y)<s_{k+1}$ for some $k$ then 
\begin{align*}
\sum_{n=1}^{\infty}\delta_n(\varphi_n(x),\varphi_n(y))^q&\ge\sum_{n=1}^{k}\delta_n(\varphi_n(x),\varphi_n(y))^q\\
&\ge k\eta^q\\
\end{align*}

if $d(x,y)$ is small enough then 
\begin{align*}
\sum_{n=1}^{\infty}\delta_n(\varphi_n(x),\varphi_n(y))^q&\ge\sum_{n=1}^{\infty}\mu_n^q \xi(d(x,y))^{q}\\
&\ge\left(\sum_{n=1}^{\infty}\mu_n^q\right) \xi(d(x,y))^{q}\\
\end{align*}
It remains to remark that if $0<q\le 1$ then $$\Delta_p(\phi(x),\phi(y))=\sum_{n=1}^{\infty}\delta_n(\varphi_n(x)-\varphi_n(t_0),\varphi_n(y)-\varphi_n(t_0))^q$$ and if $1\le q<\infty$, $$\Delta_q(\phi(x),\phi(y))=\left(\sum_{n=1}^{\infty}\delta_n(\varphi_n(x)-\varphi_n(t_0),\varphi_n(y)-\varphi_n(t_0))^q\right)^{1/q}.$$

\end{proof}

The quality of the expansion modulus is reflected in the behavior of the function $\gamma$ and estimates on the compression modulus are encode by the growth rate of the sequence $(s_n)_{n\ge 1}$ for large distances and the behavior of the function $\xi$ for small distances.
\begin{footnote}{The latter fact was already noticed in \cite{Nowak2007}.}
\end{footnote}
The maps $\varphi_n$ will be referred to as the \textit{fundamental maps} in the sequel. Denote $s^-$ the generalized inverse function of the canonical continuous function obtained by piecewise extension of $s(n)$.
  
\begin{proposition}\label{Bembeddings}
Let $q\ge 1$. Assume that $0<\xi(t)\le \gamma(t)<\infty$ for all $t>0$ and $\lim_{t\to0}\gamma(t)=0$. Then $\phi$ is a strong embedding from $\MM$ into $\ds\left(\sum_{n=1}^{\infty} X_n\right)_{q}$ so that $$\xi(d(x,y))\lesssim_s \Delta_q(\phi(x),\phi(y))\lesssim \gamma(d(x,y))$$ and 
$$s^-(d(x,y))^{1/q}\lesssim_l \Delta_q(\phi(x),\phi(y))\lesssim \gamma(d(x,y)).$$
\end{proposition}
\begin{proof}
It follows from Lemma \ref{stronggluing} that $$\Delta_q(\phi(x),\phi(y))\le \gamma(d(x,y)) \textrm{ for all } x,y\in \MM$$ and
$$\Delta_q(\phi(x),\phi(y))\ge \xi(d(x,y)) \textrm{ whenever } d(x,y) \textrm{ is small enough}$$ 
On the other hand,
$$\Delta_q(\phi(x),\phi(y))\ge \eta k^{1/q} \textrm{ whenever }s(k)\le d(x,y)\le s(k+1),$$
but $k\ge s^-(d(x,y))-1\gtrsim s^-(d(x,y))$ whenever $d(x,y)$ is large enough.
\end{proof}

\medskip

One singles out the following two typical and important regimes. If $(s_n)_{n\ge 1}$ grows at most exponentially, $i.e.\ s_n\lesssim 2^n$. Then $$\log(d(x,y))^{1/q}\lesssim_l \Delta_q(\phi(x),\phi(y))\lesssim \gamma(d(x,y)). $$
If $(s_n)_{n\ge 1}$ grows at most polynomially like $n^\xi$ for some $\xi>0$, i.e. $s_n\lesssim n^\xi$. Then $$d(x,y)^{1/(q\xi)}\lesssim_l \Delta_q(\phi(x),\phi(y))\lesssim \gamma(d(x,y)). $$

\medskip

The following proposition, whose proof is left to the reader, takes care of the $\ell_q$-sums with $0<q\le 1$.
\begin{proposition}\label{Fembeddings}
Let $0<q\le 1$. Under the same assumptions $\phi$ is a strong embedding from $\MM$ into $\ds\left(\sum_{n=1}^{\infty} X_n\right)_{q}$ so that $$\xi(d(x,y))^{q}\lesssim_s \Delta_q(\phi(x),\phi(y))\lesssim \gamma(d(x,y))^{q}$$ and 
$$s^-(d(x,y))\lesssim_l \Delta_q(\phi(x),\phi(y))\lesssim \gamma(d(x,y))^{q}.$$
\end{proposition}

\bigskip

\subsection{Strong deformation gaps for Hilbert spaces}
\subsubsection{Constructing the fundamental maps for Hilbert spaces}

When the domain space is a Hilbert space one is able to produce a sequence of fundamental maps using classical Hilbertian Theory. Two different ways for building the fundamental maps are given.

\begin{lemma}\label{schoenberg}
Let $H$ be a Hilbert space. There exists a Hilbert space $\mathcal{H}$ such that for each $r>0$ there exists a map $\psi_r\colon H\to \mathcal{H}$ such that $$\Vert\psi_r(x)-\psi_r(y)\Vert_{\mathcal{H}}=\sqrt{2\left(1-e^{-r\Vert x-y\Vert^2_H}\right)}$$\end{lemma}

\begin{proof}\ \\
$\bullet$ First approach:
We define $Exp(H)=\R\oplus H\oplus (H\otimes H)\oplus\cdots\oplus H^{\otimes n}\oplus\cdots$\newline
Let $E\colon H\to Exp(H)$, defined by $$E(x)=1\oplus x\oplus (\frac{1}{\sqrt{2!}}x\otimes x)\oplus\cdots\oplus (\frac{1}{\sqrt{n!}}x^{\otimes n})\oplus\cdots$$

\noindent First of all $Exp(H)$ is a Hilbert space and $<E(x),E(y)>=e^{<x,y>}$, hence $\Vert E(x)\Vert_{Exp(H)}=e^{\frac{1}{2}\Vert x\Vert_H^2}$. Let $\psi_r\colon H\to\mathcal{H}=Exp(H)$ defined by $$\psi_r(x)=e^{-r\Vert x\Vert^2_H} E(\sqrt{2r}x).$$
We have that 
\begin{align*}
<\psi_r(x),\psi_r(y)>_{\mathcal{H}}&=<e^{-r\Vert x\Vert^2_H} E(\sqrt{2r}x),e^{-r\Vert y\Vert^2_H} E(\sqrt{2r}y)>_{Exp}\\
											 &=e^{-r(\Vert x\Vert^2_H+\Vert y\Vert^2_H)}< E(\sqrt{2r}x),E(\sqrt{2r}y)>_{Exp}\\
											 &=e^{-r(\Vert x\Vert^2_H+\Vert y\Vert^2_H)}e^{<\sqrt{2r}x,\sqrt{2r}y>_{H}}\\
											 &=e^{-r(\Vert x\Vert^2_H-2<x,y>_{H}+\Vert y\Vert^2_H)}\\
											 &=e^{-r\Vert x-y\Vert^2_H}\\
\end{align*}
Finally remark that $\Vert\psi_r(x)\Vert_\mathcal{H}=1$.

\medskip

\noindent$\bullet$ Second approach: $u(x,y)=\Vert x-y\Vert_H^2$ is a negative kernel on $H$ and therefore for every $r>0$, $w(x,y)=e^{-r\Vert x-y\Vert_H^2}$ is a positive definite kernel on $H$. A classical theorem of Schoenberg provides a Hilbert space $\mathcal{H}$ and a mapping 
$\psi_r\colon H\to S_{\mathcal{H}}$ such that $$<\psi_r(x),\psi_r(y)>_\mathcal{H}=e^{-r\Vert x-y\Vert_H^2}.$$ 

Therefore \begin{align*}
\Vert\psi_n(x)-\psi_n(y)\Vert^2_{\mathcal{H}}&=<\psi_n(x)-\psi_n(y),\psi_n(x)-\psi_n(y)>_{\mathcal{H}}\\
											 &=\Vert \psi_n(x)\Vert^2_{\mathcal{H}}+\Vert \psi_n(y)\Vert^2_{\mathcal{H}}-2<\psi_n(x),\psi_n(y)>_{\mathcal{H}}\\
											 &=2(1-e^{-r_n\Vert x-y\Vert^2_H})\\
\end{align*}					
\end{proof}

\subsubsection{Quantitative strong embeddings of Hilbert spaces into Lebesgue sequence spaces}\ \\
Since the fundamental maps take their values in the unit sphere of a Hilbert space, using the Mazur maps it is possible to transfer the embedding into the other Lebesgue sequence spaces. Recall the following Mazur map estimates whose proofs can be found in the Appendix.

\medskip

\noindent Let $M_{2,q}\colon \ell_2 \to \ell_q$ be the classical Mazur map. Then for every pair of points in the unit sphere of $\ell_2$ one has
\begin{enumerate}
\item $2\le q<\infty$, $$ \Vert x-y\Vert_2\lesssim \Vert \M_{2,q}(x)-M_{2,q}(y)\Vert_q\lesssim\Vert x-y\Vert_2^{2/q} $$
\item $1\le q<2<\infty$, $$ \Vert x-y\Vert_2^{2/q}\lesssim \Vert \M_{2,q}(x)-M_{2,q}(y)\Vert_q\lesssim\Vert x-y\Vert_2 $$
\item $0< q\le 1$, $$ \Vert x-y\Vert_2^{2}\lesssim d_q(M_{2,q}(x),M_{2,q}(y))\lesssim\Vert x-y\Vert_2^{q} $$
\end{enumerate}

Three types of inequalities regarding the fundamental maps are gathered in the next lemma and will be used to give quantitative estimates for the compression and expansion moduli of the strong embeddings which will be constructed out of Proposition \ref{Bembeddings} and Proposition \ref{Fembeddings}.

\begin{lemma}\label{Hmappings}
For every sequence $(r_n)_{n\ge 1}$ with $0\le r_n\le 1$, there exist maps $\varphi_n\colon \ell_2\to S_{\ell_q}$ such that 
\begin{itemize}
\item $\omega_n(t)\lesssim r_n^{\gamma_q} t^{2\gamma_q}$ for all $t\ge 0$\\
\item $\rho_n(t) \ge \delta_q>0$ whenever $t\ge \ds\frac{1}{\sqrt{r_n}}$\\
\item $\rho_n(t) \gtrsim r_n^{\xi_q} t^{2\xi_q}$ for all $t\le 1$\\
\end{itemize}
where $(\gamma_q,\xi_q)=\left\{\begin{array}{lcl}
                       (1/q,1/2) &\textrm{ if }&2\le q<\infty\\
                       (1/2,1/q) &\textrm{ if }&1\le q<2\\
                       (q/2,1) &\textrm{ if }&0< q<1
                            \end{array}\right.$
\end{lemma}

\bigskip

\begin{proof}
Let $\psi_n$ be the fundamental maps from Lemma \ref{schoenberg} and define maps $\varphi_n\colon \ell_2\to S_{\ell_q}$ by $\varphi_n(x)=M_{2,q}(\psi_n(x))$. According to the Mazur map estimates one has,

\begin{enumerate}
\item $2\le q<\infty$, we have $$\left(1-e^{-r_n\Vert x-y\Vert^2_2}\right)^{1/2}\lesssim \Vert \varphi_n(x)-\varphi_n(y)\Vert_q\lesssim\left(1-e^{-r_n\Vert x-y\Vert^2_2}\right)^{1/q} $$
\item $1\le q<2<\infty$, $$\left(1-e^{-r_n\Vert x-y\Vert^2_2}\right)^{1/q}\lesssim \Vert \varphi_n(x)-\varphi_n(y)\Vert_q\lesssim\left(1-e^{-r_n\Vert x-y\Vert^2_2}\right)^{1/2}$$
\item $0< q\le 1$, $$1-e^{-r_n\Vert x-y\Vert^2_2}\lesssim d_q(M_{2,q}(x),M_{2,q}(y))\lesssim\left(1-e^{-r_n\Vert x-y\Vert^2_2}\right)^{q/2} $$
\end{enumerate} 

To prove the first lower estimate for the compression modulus suppose that $$1-e^{-r_n\Vert x-y\Vert^2_2} \ge \ds\frac{(e-1)}{e}>0.$$ Then 
$e^{-r_n\Vert x-y\Vert^2_2}\le 1-\ds\frac{(e-1)}{e}=\ds\frac{1}{e}$, i.e. $e\le e^{r_n\Vert x-y\Vert^2_2}$ and finally $\Vert x-y\Vert_2^2\ge \frac{1}{r_n}$. Therefore if $t\ge \frac{1}{\sqrt{r_n}}$ one has $\rho_n(t)\ge\delta_q$ for some universal constant $\delta_q>0$. 

\smallskip

The two other estimates are obtained from the two inequalities
\begin{itemize}
\item[$\triangleright$] $1-e^{-t}\le t$ for all $t\ge 0$
\item[$\triangleright$] $1-e^{-t}\ge\frac{t}{e}$ when $0\le t\le 1$.
\end{itemize}
\end{proof}

As a warm-up one shows what type of strong embedding from $\ell_2$ into $\ell_2$ can be constructed out of Proposition \ref{Bembeddings}. Denote $h_{(a,b)}$ the inverse of the function $t\mapsto t^a\log^b(t)$.
Let $r_n=\frac{1}{n\log(n)^{\beta}}$ for some $\beta>1$. From Lemma \ref{Hmappings} there are maps $\varphi_n\colon \ell_2\to S_{\ell_2}$ so that,
\begin{itemize}
\item $\omega_n(t) \lesssim \frac{t}{n^{1/2}\log(n)^{\beta/2}}$ for all $t\ge 0$\\
\item $\rho_n(t)\ge \delta_2>0$ whenever $t\ge n^{1/2}\log(n)^{\beta/2}$\\
\item $\rho_n(t) \gtrsim \frac{t}{n^{1/2}\log(n)^{\beta/2}}$ whenever $t\le 1$
\end{itemize}

Take $\eta=\delta_2$, $\epsilon_n=\mu_n=\frac{1}{n^{1/2}\log(n)^{\beta/2}}$, $s_n=\frac{1}{\sqrt{r_n}}=n^{1/2}\log(n)^{\beta/2}$ and $\xi(t)=\gamma(t)=t$. According to Proposition \ref{Bembeddings} for every $\beta>1$ there is a strong embedding $\ell_2\to \ell_2 $ such that $$h_{(1/2,\beta/2)}(\Vert x-y\Vert_2)^{1/2}\lesssim_l \Vert\phi(x)-\phi(y)\Vert_2\lesssim \Vert x-y\Vert_2$$
and 
$$\Vert x-y\Vert_2\lesssim_s \Vert\phi(x)-\phi(y)\Vert_2\lesssim \Vert x-y\Vert_2.$$

Following a similar procedure one proves:
\begin{theorem}\label{Hsdeformation}For every $0<q<\infty$ and $\beta>1$ there exists a strong embedding $\phi$ from $\ell_2$ into $\ell_q$ so that if  

\medskip

\begin{itemize}
\item $q\ge 2$, $$h_{(1/2,\beta/2)}(\Vert x-y\Vert_2)^{1/q}\lesssim_l \Vert\phi(x)-\phi(y)\Vert_q\lesssim \Vert x-y\Vert_2^{2/q}$$ and $$\Vert x-y\Vert_2\lesssim_s \Vert\phi(x)-\phi(y)\Vert_q\lesssim \Vert x-y\Vert_2^{2/q}.$$

\smallskip

\item $1\le q\le 2$, $$h_{(1/q,\beta/q)}(\Vert x-y\Vert_2)^{1/q}\lesssim_l \Vert\phi(x)-\phi(y)\Vert_q\lesssim \Vert x-y\Vert_2$$ and $$\Vert x-y\Vert_2^{2/q}\lesssim_s \Vert\phi(x)-\phi(y)\Vert_q\lesssim \Vert x-y\Vert_2.$$

\smallskip

\item $0< q\le 1$,  
$$h_{(1/q^2,\beta/q^2)}(\Vert x-y\Vert_2)\lesssim_l d_q(\phi(x),\phi(y))\lesssim \Vert x-y\Vert_2^{q^2}$$ and $$\Vert x-y\Vert_2^{2q}\lesssim_s d_q(\phi(x),\phi(y))\lesssim \Vert x-y\Vert_2^{q^2}.$$

\end{itemize}
\end{theorem}

\bigskip

\begin{proof}Let $\beta>1$. The proof in the case $0<q\le 1$ shall be detailed. For the remaining cases only the choice of the parameters shall be given.
\begin{itemize} 
\item For $0<q\le 1$ take $r_n=\frac{1}{n^{2/q^2}\log(n)^{(2\beta)/q^2}}$. From Lemma \ref{Hmappings} one gets fundamental maps $\varphi_n\colon \ell_2\to S_{\ell_q}$ satisfying
  
\begin{itemize}
\item $\omega_n(t)\lesssim \ds\frac{t^q}{n^{1/q}\log(n)^{\beta/q}}$\\
\item $\rho_n(t)\ge\delta_q$ whenever $t\ge n^{1/q^2}\log(n)^{\beta/q^2}$\\
\item $\rho_n(t)\gtrsim \ds\frac{t^2}{n^{2/q^2}\log(n)^{(2\beta)/q^2}}$ whenever $t\le 1$\\
\end{itemize}

Applying Proposition \ref{Fembeddings} with the parameters $\eta=\delta_q>0$,\\ $\ds\epsilon_n=\frac{1}{n^{1/q}\log(n)^{\beta/q}}$, $\ds\mu_n=\frac{1}{n^{2/q^2}\log(n)^{(2\beta)/q^2}}$, $\ds s_n=n^{1/q^2}\log(n)^{\beta/q^2}$, $\gamma(t)=t^q$ and $\xi(t)=t^2$ one gets the desired embedding.\\

\item If $2\le q$ pick $r_n=\ds\frac{1}{n\log(n)^{\beta}}$. Then
\begin{itemize}
\item $\omega_n(t)\lesssim \ds\frac{t^{2/q}}{n^{1/q}\log(n)^{\beta/q}}$\\
\item $\rho_n(t)\ge\delta_q$ whenever $t\ge n^{1/2}\log(n)^{\beta/2}$\\
\item $\rho_n(t)\gtrsim \ds\frac{t}{n^{1/2}\log(n)^{\beta/2}}$ whenever $t\le 1$\\
\end{itemize}

\item When $1\le q\le2$ choose $r_n=\frac{1}{n^{2/q}\log(n)^{(2\beta)/q}}$
\begin{itemize}
\item $\omega_n(t)\lesssim \ds\frac{t}{n^{1/q}\log(n)^{\beta/q}}$\\
\item $\rho_n(t)\ge\delta_q$ whenever $t\ge n^{1/q}\log(n)^{\beta/q}$\\
\item $\rho_n(t)\gtrsim \ds\frac{t^{2/q}}{n^{2/q^2}\log(n)^{(2\beta)/q^2}}$ whenever $t\le 1$\\
\end{itemize}
 
\end{itemize}
\end{proof}

\subsection{Applications}
\subsubsection
{$\ell_p$-compression of Lebesgue sequence spaces}\ \\

Theorem \ref{Hsdeformation} has a nice application to the $\ell_p$-compression of Lebesgue sequence spaces. First remark that all the embeddings are Lipschitz since for $t\ge 1$, $t^{a}\le t$ when $0<a\le 1$. The asymptotics of the function $h_{(a,b)}$ are needed to estimate the compression exponent. For every $c>0$ one has $t^a\log(t)^b\le t^{a+c}$ for $t$ large enough. The inverse of a increasing function being increasing the inequality $t^{1/(a+c)}\le h_{(a,b)}(t)$ holds. For instance, the compression function for the second embedding of theorem \ref{Hsdeformation} behaves asymptotically like $t^{1/(1+qc)}$ for every $c>0$ and this implies that the compression is $1$. Proceeding similarly for the remaining cases one can show:
\begin{corollary}[$\ell_q$-compression of $\ell_2$]\label{Hcompression}\ \\
\begin{itemize}
\item For $1\le q\le 2$, $\alpha_{\ell_q}(\ell_2)=1$\\
\item For $0<q\le 1$, $\alpha_{\ell_q}(\ell_2)\ge q^2$\\
\item For $2\le q<\infty$, $\alpha_{\ell_q}(\ell_2)\ge \frac{2}{q}$
\end{itemize}
\end{corollary}

In \cite{KaltonRandrianarivony2008} it is proven that $\ell_p$ does not admit a bi-Lipschitz embedding for large distances into $\ell_q$ for any $p\neq q\in[1,\infty)$. Nevertheless the best $\ell_q$-compression for $\ell_2$ can still be achieved when $1\le q\le 2$. None of the lower bounds on the compression are attained using the techniques from this article, however the lower bound $\alpha_{\ell_p}(\ell_2)\ge \frac{2}{p}$ for $2<p<\infty$, matches the one obtained in \cite{AlbiacBaudier2013} where it is proved that it is actually attained. Indeed using completely different techniques it is shown that there exists a bi-Lipschitz embedding of $(\ell_2, \Vert\cdot\Vert_2^{2/q})$ into $\ell_q$ when $q>2$. This bi-Lipschitz embedding of a snowflaking of the Euclidean distance is simultaneously a coarse and uniform embedding, i.e. a strong embedding. The strong deformation gap of this embedding is $[t^{2/q}, t^{2/q}]$ and it implies that $\alpha_{\ell_q}(\ell_2)\ge \frac{2}{q}$. More generally one has $\alpha_{\ell_q}(\ell_p)\ge \frac{1}{q}$ if $0<p\le 1\le q$, $\alpha_{\ell_q}(\ell_p)=1$ when $0<p<q\le 1$ and $\alpha_{\ell_q}(\ell_p)\ge \frac{p}{q}$ whenever $1\le p<q$. Combining those estimates together with corollary \ref{Hcompression} one obtains: 

\begin{corollary}[$\ell_q$-compression of $\ell_p$ for $0<q<p<2$]\ \\
\begin{itemize}
\item For $0< q\le 1\le p<2$, $\alpha_{\ell_q}(\ell_p)\ge \ds\frac{pq^2}{2}$\\
\item For $0<q<p\le 1$, $\alpha_{\ell_q}(\ell_p)\ge q^2/2$\\
\item For $1\le q<p<2$, $\alpha_{\ell_q}(\ell_p)\ge \frac{p}{2}$
\end{itemize}
\end{corollary}

\begin{proof}
Follows from Corollary \ref{Hcompression} together with the estimates recalled above and the fact that if $$X\coarse Y\coarse Z$$ then $\alpha_Z(X)\ge \alpha_Y(X)\alpha_Z(Y)$.
\end{proof}

For the remaining cases the compression is either 0 when there is no coarse embedding (cf \cite{JohnsonRandrianarivony2006} and \cite{MendelNaor2008}) or will be given in Section \ref{upper bounds}.

\subsubsection{Nonlinear embeddings between classical Lebesgue spaces in the range $0<p\le 1$}\ \\

\smallskip

Recall that there is \textit{no nontrivial Lipschitz map} from $\ell_p$ into $\ell_q$ if $0<q<p\le 1$ and it was left open in \cite{AlbiacBaudier2013} whether or not one could construct a weaker nonlinear embedding. One can easily derive from Theorem \ref{Hsdeformation} and the table in Section 6 from \cite{AlbiacBaudier2013} that there exists a strong embedding between those spaces.

\begin{corollary}\label{Fstrongpq}Let $0<q<p\le 1$. For every $\beta>1$ there exists a strong embedding $\phi$ from $\ell_p$ into $\ell_q$ so that  
$$h_{\left(\frac{1}{q^2},\frac{\beta}{q^2}\right)}(d_p(x,y))^{1/2}\lesssim_l d_q(\phi(x),\phi(y))\lesssim d_p(x,y)^{q^2/2}$$ and $$d_p(x,y)^{q}\lesssim_s d_q(\phi(x),\phi(y))\lesssim d_p(x,y)^{q^2/2}.$$
\end{corollary}

\bigskip

It was known that for $1\le p\le 2$ the space $L_p$ embeds uniformly \cite{BenyaminiLindenstrauss2000} or coarsely \cite{Nowak2006} into $\ell_p$ and actually strongly from \cite{Kraus}. These embeddings can now be extended to values of $p$ less than 1 and can be made quantitative. 
 \begin{corollary}\label{Lplp}Let $0<p<2$. For every $\beta>1$ there exists a strong embedding $\phi$ from $L_p$ into $\ell_p$ such that if   

\medskip

\begin{itemize}
\item $0< p\le 1$,  
$$h_{\left(\frac{1}{p^2},\frac{\beta}{p^2}\right)}(d_{L_p}(x,y))^{1/2}\lesssim_l d_{\ell_p}(\phi(x),\phi(y))\lesssim d_{L_p}(x,y)^{\frac{p^2}{2}}$$ and $$d_{L_p}(x,y)^{p}\lesssim_s d_p(\phi(x),\phi(y))\lesssim d_{L_p}(x,y)^{\frac{p^2}{2}}.$$

\medskip

\item $1\le p< 2$, $$h_{(\frac{1}{p},\frac{\beta}{p})}(\Vert x-y\Vert_{L_p})^{1/2}\lesssim_l \Vert\phi(x)-\phi(y)\Vert_{\ell_p}\lesssim \Vert x-y\Vert_{L_p}^{p/2}$$ and $$\Vert x-y\Vert_{L_p}\lesssim_s \Vert\phi(x)-\phi(y)\Vert_{\ell_p}\lesssim \Vert x-y\Vert_{L_p}^{p/2}.$$
\end{itemize}
\end{corollary}
\begin{proof}
Use the quantitative estimates corresponding to the chains of embeddings $$L_p\isometricword L_1\strong L_2\strong \ell_p \textrm{ if } 0<p\le 1$$ and 
$$L_p\strong L_2\strong \ell_p \textrm{ when }1\le p\le 2.$$
\end{proof}
Strong embeddings between Orlicz sequence spaces are studied in \cite{Kraus}. For instance if the Orlicz spaces are taken to be $\ell_q$-spaces, by inspecting the proofs a coarse deformation gap $[t^{1/q},t^{2/q})$ and a uniform deformation gap $(t^{2/q},t^{1/q}]$ are obtained while embedding $\ell_2$ into $\ell_q$ for $2<q$. Reformulating Theorem \ref{Hsdeformation} in terms of nonlinear gaps one obtains a coarse deformation gap $(t^{2/q},t^{2/q}]$ and a uniform deformation gap $[t,t^{2/q}]$. These gaps have to be compared to the strong deformation gap $[t^{2/q},t^{2/q}]$ that can be found in \cite{AlbiacBaudier2013}.

\begin{remark}One can derive the following interesting compression estimates out of Corollary \ref{Fstrongpq} and the asymptotics of the function $h$: 

\begin{itemize}
\item[$\triangleright$] For $1\le p<2$, $\alpha_{\ell_p}(L_p)\ge \frac{p}{2}$\\
\item[$\triangleright$] For $0<p\le 1$, $\alpha_{\ell_p}(L_p)\ge \frac{p^2}{2}$\\
\end{itemize}
\end{remark}

\subsubsection{Embeddings of metrics of Lipschitz-negative type}\ \\

\smallskip

A metric is of negative type if and only if its $1/2$-snowflaking (the square root of the distance) admits an isometric embedding into a Hilbert space. Low distortion embeddings of finite metric spaces of negative type into $L_1$ or $\ell_2$ is of great interest in Theoretical Computer Science. Roughly speaking, the worst approximation ratio achieved by a semidefinite programming relaxation for the general Sparsest Cut problem coincides with the $L_1$-distortion of metrics of negative type. It is worth mentioning that $L_1$ is a metric of negative type and so is every metric space which is isometric to a subset of $L_1$. Goemans and Linial raised the following conjecture: Is every metric space of negative type bi-Lipschitz equivalent to a subset of $L_1$? This conjecture was disproved by Khot and Vishnoi \cite{KhotVishnoi2005}. See also \cite{LeeNaor2006} for a counter-example not related to the Unique Game Conjecture. A metric will be said to be of \textit{Lipschitz-negative type} if its $1/2$-snowflaking admits a bi-Lipschitz embedding into a Hilbert space. In other words, $(\MM,d)$ is a metric of Lipschitz-negative type if there exist $D>0$, a Hilbert space $H$ and $f\colon \MM\to H$ such that $$\frac{1}{D}\sqrt{d(x,y)}\le \Vert f(x)-f(y)\Vert_H\le \sqrt{d(x,y)}.$$

Those metrics are called $D$-half snowflakes in \cite{LeeMoharrami2010} where they prove that a metric of Lipschitz negative type is not necessarily bi-Lipschitz equivalent to a metric of (isometric) negative type. From \cite{AlbiacBaudier2013} it can be derived that every (non-uniformly discrete and unbounded) metric of Lipschitz-negative type admits a strong embedding into $\ell_q$, $q>2$ with strong deformation gap $[t^{1/q},t^{1/q}]$.  Another application of Theorem \ref{Hsdeformation} is stated in the next corollary in a weak form for the sake of clarity. The compression modulus is actually a little better.

\begin{corollary}\label{lipneg}
Every (non-discrete and unbounded) metric of Lipschitz-negative type admits a strong embedding into $\ell_q$ with 
\begin{itemize}
\item uniform deformation gap $\left[t^{q},t^{\frac{q^2}{2}}\right]$ and coarse deformation gap $\left(t^{\frac{q^2}{2}},t^\frac{q^2}{2}\right]$ if $0< q\le 1$.

\smallskip

\item uniform deformation gap $\left[t^{\frac{1}{q}},t^{\frac{1}{2}}\right]$ and coarse deformation gap $\left(t^{\frac{1}{2}},t^{\frac{1}{2}}\right]$ if $1\le q< 2$.

\smallskip

\item uniform deformation gap $\left[t^{\frac{1}{2}},t^{\frac{1}{q}}\right]$ and coarse deformation gap $\left(t^{\frac{1}{q}},t^{\frac{1}{q}}\right]$ if $q>2$.
\end{itemize}
\end{corollary} 

It follows from the embedding theorem of Assouad \cite{Assouad1983} that a doubling metric is a metric of Lipschitz-negative type and Assouad's embedding provides strong embeddings into the Banach spaces $\ell_q$ with strong deformation gap $(t,t)$ which are better than the gaps in Corollary \ref{lipneg}. However it is not clear if one can get better estimates in the zone $0<q<1$. An interesting example of a non-uniformly discrete and unbounded metric of Lipschitz negative type is the (3-dimensional) real Heisenberg group equipped with its intrinsic Carnot-Carath\'{e}odory metric. Indeed in \cite{LeeNaor2006} it is shown that the real Heisenberg group admits an equivalent metric of negative type.

\subsection{A purely coarse embedding}\ \\
Recall the original Dadarlat-Guentner criterion.
\begin{proposition}[Dadarlat-Guentner \cite{DadarlatGuentner2003}] Let $X$ be a metric space.  Then $X$ is coarsely embeddable into a Hilbert space if and only if for every $R>0$ and $\epsilon>0$ there exists a Hilbert space valued map $\xi\colon X\to \mathcal{H}$ such that $\Vert \xi(x)\Vert=1$ for all $x\in X$ and such that 
\begin{enumerate}
\item $\omega_\xi(R)\le\epsilon$
\item $\lim_{t\to\infty}\rho_\xi(t)=2$
\end{enumerate}
\end{proposition}

\begin{lemma}\label{coarsegluing} Let $(\MM,d)$ be a metric space. Suppose that $(X_n,\delta_n)$, $n\ge 1$ is a translation invariant metric linear spaces and suppose further that one can find $\eta>0$, sequences $(r_n)_{n\ge 1}$, $(\epsilon_n)_{n\ge 1}$, $(s_n)_{n\ge 1}$ and maps $\varphi_n\colon \MM\to S_{X_n}\colon=\{x\in X_n;\vert x\vert=1\}$ satisfying 
\begin{enumerate}
\item $(\epsilon_n)_{n\ge 1}$ is $q$-summable; $(r_n)_{n\ge 1}$ and $(s_n)_{n\ge 1}$ are nondecreasing and unbounded\\
\item $\delta_n(\varphi_n(x),\varphi_n(y))\le \epsilon_n$ whenever $d(x,y)\le r_n$\\
\item $\delta_n(\varphi_n(x),\varphi_n(y))\ge \eta$ whenever $d(x,y)\ge s_n$
\end{enumerate}

\medskip

Fix a base point $t_0\in\MM$ and define the map $\phi\colon\MM\to \left(\sum_{n=1}^\infty X_n\right)_{q}$ by $\phi(x)=(\varphi_n(x)-\varphi_n(t_0))_{n\ge 1}$. Then there exists a constant $K> 0$ such that the following inequalities hold:

\medskip

\begin{itemize}

\item if $0<q\le 1$.
\begin{equation*}
\Delta_q(\phi(x),\phi(y))\le 2^qk+K \textrm{ whenever }r_k\le d(x,y)\le r_{k+1}
\end{equation*} 

and
\begin{equation*}
\Delta_q(\phi(x),\phi(y))\ge k\eta^q \textrm{ whenever }s_k\le d(x,y)\le s_{k+1}.
\end{equation*}

\item if $q\ge 1$.
\begin{equation*}
\Delta_q(\phi(x),\phi(y))\le 2k^{1/q}+K^{1/q} \textrm{ whenever }r_k\le d(x,y)\le r_{k+1}
\end{equation*} 

and
\begin{equation*}
\Delta_p(\phi(x),\phi(y))\ge \eta k^{1/q} \textrm{ whenever }s_k\le d(x,y)\le s_{k+1}.
\end{equation*}
\end{itemize}
\end{lemma}

The next two propositions are analogues of Proposition \ref{Fembeddings} and Proposition \ref{Bembeddings} .

\begin{proposition}\label{Bembeddingc}
Let $q\ge 1$. Assume that there are $\delta> 0$ and sequences as in Lemma \ref{coarsegluing}. Then there is a coarse embedding  $\phi$ from $\MM$ into $\left(\sum_{n=1}^\infty X_n\right)_{q}$ so that $$s^-(d(x,y)))^{1/q}\lesssim_l \Delta_q(\phi(x),\phi(y))\lesssim_l r^-(d(x,y))^{1/q}.$$ 
\end{proposition}
\begin{proof}
It follows from Lemma \ref{coarsegluing} that $$\Delta_q(\phi(x),\phi(y))\le 2k^{1/q}+K^{1/q} \textrm{ whenever } r(k)\le d(x,y)\le r(k+1)$$
Hence $\Delta_q(\phi(x),\phi(y))\lesssim r^-(d(x,y))^{1/q}$ whenever $d(x,y)$ is large enough. On the other hand,
$$\Delta_q(\phi(x),\phi(y))\ge \eta k^{1/q} \textrm{ whenever }s_k\le d(x,y)\le s_{k+1},$$
but $k\ge s^-(d(x,y))-1\gtrsim s^-(d(x,y))$ whenever $d(x,y)$ is large enough. \\It implies that $\Delta_q(\phi(x),\phi(y))\gtrsim s^-(d(x,y))^{1/q}.$
\end{proof}

\begin{proposition}\label{Fembeddingc}
Let $0<q\le 1$. Assume that there are $\delta> 0$ and sequences as in Lemma \ref{coarsegluing}. Then there is a coarse embedding  $\phi$ from $\MM$ into $\left(\sum_{n=1}^\infty X_n\right)_{q}$ so that

$$s^-(d(x,y))\lesssim_l \Delta_q(\phi(x),\phi(y))\lesssim_l r^-(d(x,y)).$$
In particular, if $(s_n)_{n\ge 1}$ and $(r_n)_{n\ge 1}$ grow linearly, $\phi$ is a Lipschitz embedding for large distances.  
\end{proposition}

\begin{proof}
It follows from Lemma \ref{coarsegluing} that $$\Delta_q(\phi(x),\phi(y))\le 2^qk+K \textrm{whenever }r_k\le d(x,y)\le r_{k+1}$$
Hence $\Delta_q(\phi(x),\phi(y))\lesssim r^-(d(x,y))$ whenever $d(x,y)$ is large enough. On the other hand,
$$\Delta_q(\phi(x),\phi(y))\ge k\eta^q \textrm{ whenever }s_k\le d(x,y)\le s_{k+1},$$
but $k\ge s^-(d(x,y))-1\gtrsim s^-(d(x,y))$ whenever $d(x,y)$ is large enough.
\end{proof}

The influence of the parameter $\delta$ and the sequence $(\epsilon_n)_{n\ge 1}$ is weak in this construction. A careful choice of these parameters can decrease the constants involved at certain steps in the proofs and can improve the threshold for which the inequalities are valid. This purely coarse embedding gives different coarse deformation gaps and weaker lower bounds on the compression. Combining Proposition \ref{Bembeddingc} and Lemma \ref{Hmappings} it can be shown that it gives a coarse embedding from $\ell_2\to \ell_2 $ with coarse deformation gap $(t^{1/3},t^{1/2}]$. Indeed, choose the parameters to be $\eta=\ds\sqrt{\frac{2(e-1)}{e}}>0$, $r_n=n$, $\epsilon_n=n^{-\nu}$ for some $\nu>0$, $t_n=\ds\left(\frac{\epsilon_n}{r_n}\right)^2$ and  $s_n=\ds\frac{1}{\sqrt{t_n}}=\frac{r_n}{\epsilon_n}=n^{1+\nu}$ with $\nu>1/2$ to ensure summability. 
 
\bigskip

\begin{remark}Using the Mazur maps one obtains a coarse embedding from $\ell_2$ into $\ell_q$ which has the following coarse deformation gaps 
\begin{itemize}
\item[$\triangleright$] $(t^{\frac{q^2}{1+q^2}},t]$ if $0<q\le 1$\\
\item[$\triangleright$] $(t^{\frac{1}{1+q}},t^\frac{1}{q}]$ if $1\le q<2$
\item[$\triangleright$] $(t^{\frac{2}{3q}},t^\frac{1}{q}]$ if $1\le q<2$
\end{itemize}

Since $t^{1/q}\le t$ for $t\ge 1$, one gets lower estimates on the compression which are weaker than the one obtained in the preceding sections. Namely,
\begin{itemize}
\item[$\triangleright$] For $0<q\le 1$, $\alpha_{\ell_q}(\ell_2)\ge \frac{q^2}{1+q^2}$\\
\item[$\triangleright$] For $1\le q<2$, $\alpha_{\ell_q}(\ell_2)\ge \frac{1}{1+q}$\\
\item[$\triangleright$] For $q>2$, $\alpha_{\ell_q}(\ell_2)\ge \frac{2}{3q}$
\end{itemize}
\end{remark}

\subsection{Upper bounds on the Lebesgue-compressions of $\ell_p$}\label{upper bounds}\ \\

\smallskip

In this section two different techniques are used to compute upper bounds for the compression exponent. The first one relies on works from Kalton and Randrianarivony \cite{KaltonRandrianarivony2008} and requires the introduction of a special family of graphs. The second technique appeared in a paper of T. Austin \cite{Austin2011} and shows a delicate interplay between the distortion of certain graphs and the compression.

\subsubsection{Computing $\alpha_{\ell_q}(\ell_p)$}\ \\

Denote $G_k(\N)$ the set of $k$-element subsets of $\N$ equipped with the distance $\rho(A,B)=\frac{\vert A\Delta B\vert}{2}$.

\begin{theorem}[Kalton-Randrianarivony \cite{KaltonRandrianarivony2008}]\label{KR}
Let $Y$ be a reflexive Banach space so that there exists $r\in(1,\infty)$ with the following property. If $y\in Y$ and $(y_n)_{n\ge1}$ is a weakly null sequence in $Y$ then 
\begin{equation}\label{smooth}\lim\sup\Vert y+y_n\Vert^r\le \Vert y\Vert^r+\lim\sup\Vert y_n\Vert^r.
\end{equation}
Assume now the $\Mb$ is an infinite subset of $\N$ and $f\colon G_k(\Mb)\to Y$ is a Lipschitz map. Then for any $\epsilon>0$, there exists an infinite subset $\Mb'$ of $\Mb$ such that: $$\diam f(G_k(\Mb'))\le 2\omega_f(1)k^{1/r}+\epsilon.$$
\end{theorem}

The fact that equation \ref{smooth} is satisfied with $r=q$ for all the reflexive sequence spaces $\ell_q$ is the key point in estimating an upper bound for the compression.
\begin{corollary}[$\ell_q$-compression of $\ell_p$]\label{alphapq}\ \\
Let $1\le p<q<\infty$. Then $$\alpha_{\ell_q}(\ell_p)=\frac{p}{q}.$$

\noindent If $0<p\le 1<q<\infty$. Then $$\alpha_{\ell_q}(\ell_p)=\frac{1}{q}.$$
\end{corollary}

\begin{proof} Let $f\colon \ell_p\to \ell_q$ such that $$d_p(x,y)^{\alpha}\lesssim \Vert f(x)-f(y)\Vert_q\lesssim d_p(x,y) \textrm{ whenever } d_p(x,y)\ge1.$$
Denote $(e_n)_{n\ge 1}$ the canonical basis of $\ell_p$ and consider the map $\varphi(u)=e_{u_1}+\dots+e_{u_k}$ where $u=(u_1,\dots,u_k)\in G_k(\N)$. It is clear that the map $\varphi$ is $2$-Lipschitz and $\varphi(G_k(\N))$ is a 1-discrete subset of $\ell_p$, therefore $\omega_{f\circ\varphi}(1)\lesssim 1$. By Theorem \ref{KR} there is an infinite subset $\Mb$ of $\N$ such the $\diam(f\circ \varphi)(G_k(\Mb))\lesssim k^{1/q}$. But for $p\ge 1$ $\diam(f\circ\varphi)(G_k(\Mb))\gtrsim \diam(\varphi(G_k(\Mb))^\alpha\gtrsim k^{\alpha/p}$. It implies that $\alpha\le p/q$. One uses the inequality $\diam(f\circ\varphi)(G_k(\Mb))\gtrsim k^{\alpha}$ for $0<p\le 1$. The equalities follow from \cite{AlbiacBaudier2013}.
\end{proof}

The same idea can be applied for functions spaces in some cases. Indeed, since equation \ref{smooth} holds with $r=q$ when $1<q\le 2$ for the function spaces $L_q$ one can get the upper estimate $\alpha_{L_q}(L_p)\le\alpha_{L_q}(\ell_p)\le\frac{p}{q}$ when $1\le p<q\le 2$. Note that $\alpha_{L_q}(L_p)\ge \frac{p}{q}$ from \cite{MendelNaor2004}, hence $\alpha_{L_q}(L_p)=\frac{p}{q}$. When $q\ge 2$ the equation \ref{smooth} is satisfied only with $r=2$ for the spaces $L_q$ and no upper bound better than one can be achieved for $p\ge 2$. Since $L_2$ embeds linearly isometrically into $L_q$ if $1\le p\le 2\le q$ one has $\alpha_{L_q}(L_p)=\frac{p}{2}$. Now the only estimate known in the case $2\le p\le q$ is $\frac{p}{q}\le \alpha_{L_q}(L_p)\le 1$. 

\mk

The same technique would also give the upper bound $\alpha_{L_q}(\ell_p)\le\frac{1}{\min\{2,q\}}$ for $0<p\le 1<q<\infty$. As one shall see there is an alternate path to prove this result.

\medskip

As previously mentioned it is possible to find upper bounds on the compression exponent for uniformly discrete metric spaces using Lemma 3.1 in Tim Austin's article \cite{Austin2011}. Recall that the \textit{distortion} of an embedding $f\colon \D\to \T$ is $$\textrm{dist}(f):=\ds\sup_{u\neq v\in \D}\frac{d_\D(u,v)}{d_\T(f(u),f(v))}\cdot\sup_{u\neq v\in \D}\frac{d_\T(f(u),f(v))}{d_\D(u,v)}$$ and $$c_\T(\D):=\sup\{\textrm{dist}(f); f\colon\D\to \T\}.$$ For our purposes the full power of Austin's lemma is not needed but only the following which is better suited for Banach spaces:

\begin{lemma}\label{Baustin}
Suppose that $X,Y$ are Banach spaces. Suppose further that we can find a sequence of finite 1-discrete metric spaces $(\NN_n,\delta_n)$ and embeddings $\iota_n\colon \NN_n\hookrightarrow X$ such that
\begin{itemize}
\item the $\NN_n$ are increasing in diameter, diam$(\NN_n,\delta_n)\to\infty$
\item the $\NN_n$ are embedded in $X$ with uniformly bounded distortion, i.e. there are some fixed $D\ge 1$ such that $$\delta_n(u,v)\le \Vert \iota_n(u)-\iota_n(v)\Vert_X\le  D\delta_n(u,v),\ \forall u,v\in\NN_n,\forall n\ge 1$$
\item the $\NN_n$ have bad distortion into $Y$, i.e. for some $\eta>0$ we have $c_Y(\NN_n,\delta_n)\gtrsim diam(\NN_n,\delta_n)^\eta$ for all $n\ge 1$
\end{itemize}
then, $$\alpha_Y(X)\le 1-\eta.$$
\end{lemma}
\begin{proof}
We may assume that $\alpha_{Y}(X)>0$. Let $\alpha<\alpha_{Y}(X)$, and let $f\colon X\to Y$ such that $$\Vert x-y\Vert^{\alpha}\lesssim \Vert f(x)-f(y)\Vert\lesssim \Vert x-y\Vert \textrm{ whenever } \Vert x-y\Vert\ge1.$$

Since $1\le\delta_n(u,v)\le \Vert \iota_n(u)-\iota_n(v)\Vert_X$ for all $u,v\in\NN_n$ and $n\ge 1$ we have
\begin{align*}
\Vert \iota_n(u)-\iota_n(v)\Vert^{\alpha}&\lesssim \Vert f\circ\iota_n(u)-f\circ\iota_n(v)\Vert\lesssim \Vert\iota_n(u)-\iota_n(v)\Vert\\
\delta_n(u,v)^{\alpha}&\lesssim \Vert f\circ\iota_n(u)-f\circ\iota_n(v)\Vert\lesssim \delta_n(u,v)\\
\end{align*}

But \begin{align*}
\textrm{dist}(f\circ\iota_n)&=\ds\max_{u\neq v\in \NN_n}\frac{\delta_n(u,v)}{\Vert f\circ\iota_n(u)-f\circ\iota_n(v)\Vert}\cdot\max_{u\neq v\in \NN_n}\frac{\Vert f\circ\iota_n(u)-f\circ\iota_n(v)\Vert}{\delta_n(u,v)}\\
 & \lesssim \max_{u\neq v\in \NN_n}\delta_n(u,v)^{1-\alpha}\\
 & \lesssim diam(\NN_n,\delta_n)^{1-\alpha}\\
\end{align*}

Hence $diam(\NN_n,\delta_n)^{1-\alpha}\gtrsim c_Y(\NN_n,\delta_n)\gtrsim diam(\NN_n,\delta_n)^\eta$ and $\alpha\le 1-\eta.$
\end{proof}

The conclusion of Lemma \ref{Baustin} still holds in a more general setting, in particular if $X$ is the metric space $(\ell_p,d_p)$ for $0<p<1$. 
In order to bound the compression, let's say  $\alpha_{L_q}(L_2)$ for instance, the strategy is to find a sequence of bi-Lipschitz copies (the distortions being bounded from above by a universal constant) of graphs with increasing diameter inside $L_2$ with large distortion into $L_q$. The Hamming cubes equipped with an $\ell_p$-like metric will be used to bound from above the $L_q$-compression of the $L_p$-spaces. Let $H_m=\{0,1\}^m$ endowed with the $\ell_p$-distance $d_p$ for $0<p<\infty$. Enflo \cite{Enflo1969} gave a lower estimate on the Euclidean distortion of the Hamming cubes with the Hamming distance. Using what is now known as the Enflo-type and following exactly the same lines as Enflo's original proof one has: 
 
\begin{theorem}[Enflo]\label{Enflo}
Let $(\MM,d)$ be a metric space with Enflo-type $q$ then
$$c_{\MM}(H_m,d_p)\gtrsim diam(H_m)^{\frac{1}{p}-\frac{1}{q}},\ p\ge 1.$$
$$c_{\MM}(H_m,d_p)\gtrsim diam(H_m)^{1-\frac{1}{q}},\ 0<p\le 1.$$
\end{theorem}

\begin{corollary}
For $0<p\le 1\le q<\infty$, $$\alpha_{L_q}(\ell_p)=\frac{1}{\min\{q,2\}}.$$
\end{corollary}

\begin{proof}

It is clear that $\ell_p$ contains isometric copies of the Hamming cubes $(H_m,d_p)$ for every $m\ge 1$. One concludes using Theorem \ref{Enflo} and invoking the fact that $L_q$ has Enflo-type $q$ if $1\le q\le 2$ and Enflo-type $2$ if $q\ge 2$. The lower bound can be found in \cite{AlbiacBaudier2013}.
\end{proof}

In \cite{AlbiacBaudier2013}, for $0<p\neq q<\infty$ the parameter $$s_{\ell_q}(\ell_p):=\sup\{s\le 1: (\ell_p, d_p^s)\lip (\ell_q, d_q)\}$$ was introduced. It was shown that $\frac{p}{q}\le s_{\ell_q}(\ell_p)\le\frac{p}{2}$ when $1\le p\le 2< q$ and $\frac{p}{q}\le s_{\ell_q}(\ell_p)\le 1$ for $2\le p< q$ and the question whether or not it was possible to close the gaps was raised. It is straightforward that the inequalities $0\le s_{\ell_q}(\ell_p)\le \alpha_{\ell_q}(\ell_p)\le 1$ hold. From the results of the current section it follows that these two parameters coincide in numerous situations and satisfy the same tight estimates given in Corollary \ref{alphapq}. 

\begin{corollary}[Estimation of the parameter $s_{\ell_q}(\ell_p)$]\ \\
Let $q>1$ and $0<p<q<\infty$. Then $$s_{\ell_q}(\ell_p)=\alpha_{\ell_q}(\ell_p).$$
\end{corollary} 

\subsubsection{Application to the coarse embeddability of $L_q$ into $\ell_q$, $q>2$}\ \\

The problem of the coarse embeddability of $L_q$ into $\ell_q$ for $q>2$ is an important open problem. So far it is still unknown whether or not there exists a coarse embedding. However a straightforward application of the upper estimate on the parameter $\alpha_{\ell_q}(\ell_2)$ in combination with the fact that $L_q$ contains an isometric copy of $\ell_2$ shows that such an embedding cannot be too good and will deteriorate as $q$ increases.

\begin{corollary}
For $2<r\le q$, $$0\le \alpha_{\ell_q}(L_r)\le \frac{2}{q}$$
\end{corollary}  
\section{Coarse deformation gaps for metric spaces with property A}
\subsection{Metric measured spaces with property A}\ \\

In this section one considers $\ell_p$-sums over sets that are not necessarily countable. For the sake of simplicity one shall restrict ourselves to the zone $p\ge 1$ even if it is clear that a parallel study can be carried over for $0<p<1$. Let $\Gamma$ be a set, $(z_\gamma)_{\gamma\in \Gamma}\in \R^{\Gamma}$ and define $\ds\sum_{\gamma\in\Gamma}\vert z_\gamma\vert^p:=\ds\sup_{J\subset \Gamma,\#J<\infty}\sum_{j\in J}\vert z_j\vert^p$. Endowed with the classical $\ell_p$-norm the linear space $$\ell_p(\Gamma,\R):=\left\{z=(z_\gamma)_{\gamma\in \Gamma}\in\R^\Gamma;\ \sum_{\gamma\in\Gamma}\vert z_\gamma\vert^p<\infty\right\}$$ is a Banach space.

\smallskip

Recall the original definition of Yu's property A. A discrete metric space with bounded geometry $(\MM,d)$ is said to have property A if for any $R>0$, $\epsilon>0$, there exists $S>0$ and a family $(A_x)_{x\in \MM}$ of finite, nonempty subsets of $\MM\times \N$ such that 
\begin{enumerate}
\item $(y,n)\in A_x$ implies $d(x,y)\le S$;
\item $d(x,y)\le R$ implies $\ds\frac{\vert A_x\Delta A_y\vert}{\vert A_x\bigcap A_y \vert}\le \epsilon$
\end{enumerate}

The connection between property A and coarse embeddings is better seen with the equivalent definitions listed below (see \cite{Tu2001}) where it is shown that having property A is equivalent to the existence of fundamental maps. 

\begin{proposition}[Equivalent definitions of property A]
Let $\MM$ be a discrete metric space with bounded geometry. The following assertions are equivalent:
\begin{enumerate}
\item $\MM$ has property A
\item $\forall\ R>0$, $\forall\ \epsilon>0$, $\exists S>0$, $\exists (\zeta_x)_{x\in\MM}$, $\zeta_x\in \ell_2(\MM\times\N,\R)$,  $\supp(\zeta_x)\subset B(x,S)\times\N$, $\Vert \zeta_x\Vert_{\ell_2(\MM\times\N,\R)}=1$ and $\Vert \zeta_x-\zeta_y\Vert_{\ell_2(\MM\times\N,\R)}\le \epsilon$ whenever $d(x,y)\le R$
\item $\forall\ R>0$, $\forall\ \epsilon>0$, $\exists S>0$, $\exists (\eta_x)_{x\in\MM}$ $\eta_x\in \ell_2(\MM,\R)$,  $\supp(\eta_x)\subset B(x,S)$, $\Vert \eta_x\Vert_{\ell_2(\MM,\R)}=1$ and $\Vert \eta_x-\eta_y\Vert_{\ell_2(\MM,\R)}\le \epsilon$ whenever $d(x,y)\le R$
\item $\forall\ R>0$, $\forall\ \epsilon>0$, $\exists S>0$, $\exists (\xi_x)_{x\in\MM}$, $\xi_x\in \ell_1(\MM,\R)$,  $\supp(\xi_x)\subset B(x,S)$, $\Vert \xi_x\Vert_{\ell_1(\MM,\R)}=1$ and $\Vert \xi_x-\xi_y\Vert_{\ell_1(\MM,\R)}\le \epsilon$ whenever $d(x,y)\le R$
\end{enumerate}
\end{proposition}
Following the same ideas one starts with what will be the main quantitative lemma which is derived from Yu's property A. 

\begin{lemma}\label{Agluing} Let $(\MM,d)$ be a metric space and $1\le p<\infty$. Denote $X=\ell_p(\MM,\R)$. Suppose that one can find sequences $(r_n)_{n\ge 1}$, $(\epsilon_n)_{n\ge 1}$, $(s_n)_{n\ge 1}$ and maps $\varphi_n\colon \MM\to S_{X}$ satisfying 
\begin{enumerate}
\item $(\epsilon_n)_{n\ge 1}$ is $p$-summable; $(r_n)_{n\ge 1}$ and $(s_n)_{n\ge 1}$ are nondecreasing and unbounded\\
\item $\Vert \varphi_n(x)-\varphi_n(y)\Vert_p\le \epsilon_n$ whenever $d(x,y)\le r_n$\\
\item $\supp(\varphi_n(x))\subset B(x,s_n)$
\end{enumerate}

\medskip

Define $\phi\colon\MM\to \ell_p(\N,X)$ by $\phi(x)=(\varphi_n(x)-\varphi_n(t_0))_{n\ge 1}$, then there exists a constant $K> 0$ such that the following inequalities hold:

\begin{equation*}
\Vert \phi(x)-\phi(y)\Vert_p\le 2k^{1/p}+K^{1/p} \textrm{ whenever }r_k\le d(x,y)\le r_{k+1}
\end{equation*} 

and
\begin{equation*}
\Vert \phi(x)-\phi(y)\Vert_p\ge 2k^{1/p} \textrm{ as long as }2s_k\le d(x,y)\le 2s_{k+1}.
\end{equation*}

\end{lemma}

\begin{proof}
Simply remark that $\varphi_n(x)$ and $\varphi_n(y)$ are disjoint if $d(x,y)>2s_n$ and apply lemma \ref{coarsegluing} with $\eta=2^q$ and $X_n=X$ for all $n$.
\end{proof}

For our purposes it is natural to extend property $A$ to metric measured spaces in the following sense. A metric space $(\MM,d)$ with a measure denoted $\vert\cdot\vert$ is said to have property $A$ if for any $R>0$, $\varepsilon>0$, there exists $S>0$ and a family $(A(x))_{x\in \MM}$ of compact, nonempty subsets of $\MM$ such that 
\begin{enumerate}
\item $z\in A(x)$ implies $d(x,z)\le S$
\item $d(x,y)\le R$ implies $\ds\frac{\vert A(x)\Delta A(y)\vert}{\vert A(x)\cap A(y)\vert}\le \varepsilon$
\end{enumerate}

A convenient terminology is needed in order to study property A in its quantitative aspect. Let $\epsilon=(\epsilon_n)_{n\ge 1}$ and $r=(r_n)_{n\ge 1}$. An \textit{$(\epsilon,r)$-A collection} $\A=(A_n(x))_{(n,x)\in\N\times\MM}$ is a collection of compact sets so that for all $n\ge 1$ the collection of sets $(A_n(x))_{x\in\MM}$ satisfies the two conditions in the definition of property $A$ with  $\varepsilon=\epsilon_n$ and $R=r_n$ for some $S_n>0$. 

\smallskip

A \textit{fast $(\epsilon,r)$-A collection} is an $(\epsilon,r)$-A collection such that $\epsilon=(\epsilon_n)_{n\ge 1}$ is summable. An $(\epsilon,r)$-A collection such that $r_n=n$ is said to be \textit{adapted}.

\smallskip

Since the sets $A_n(x)$ are compact, by condition $(1)$ of property $A$ there exists a number $0<S_n<\infty$ such that $A_n(x)$ is included in the ball centered at $x$ and of radius $S_n$. For an A-collection $\A=(A_n(x))_{(n,x)\in\N\times\MM}$ denote $rad_\A$ the smallest non-decreasing function $S$ such that for all $x$ and $n$ $$A_n(x)\subset B_{d}(x,S(n))$$ and call it the \textit{ radial dilation function } of the A-collection. Normalized characteristic functions of fast $A$-collections produce fundamental maps.

\begin{lemma}\label{lemmamenable}
If a metric measured space $\MM$ admits a fast $(\epsilon,r)$-A collection $\A$ then there exist maps \\$\varphi_n\colon \MM\to S_{\ell_p(\MM,\R)}$ satisfying 
\begin{enumerate}
\item $\Vert \varphi_n(x)-\varphi_n(y)\Vert_p\lesssim\epsilon_n^{1/p}$ whenever $d(x,y)\le r_n$\\
\item $\supp(\varphi_n(x))\subset B(x,rad_\A(n))$
\end{enumerate}
\end{lemma}

\begin{proof}
Let $\A=(A_n(x))_{(n,x)\in\N\times\MM}$ be a fast $(\epsilon,r)$-A collection and define $\varphi_n(x)=\ds\frac{\chi_{A_n(x)}}{\vert A_n(x)\vert^{1/p}}$. Clearly $\varphi_n\in \ell_p(\MM,\R)$, $\Vert \varphi_n(x)\Vert_p=1$ and $$\supp(\varphi_n(x))=A_n(x)\subset B(x,rad_{\A}(n)).$$ Following the case $p=2$ in \cite{NowakYu} one can show that $\Vert\varphi_n(x)-\varphi_n(y)\Vert_p^p\le 2\epsilon_n$ whenever $d(x,y)\le r_n$. Indeed for $p\ge 1$ the inequality $\vert x-y\vert^p\le \vert x^p-y^p\vert$ holds for any $0\le x,y\le 1$. Assume that $\vert A_n(x)\vert\le \vert A_n(y)\vert$.
\begin{align*}
\Vert\varphi_n(x)-\varphi_n(y)\Vert_p^p &=\left\Vert\frac{\chi_{A_n(x)}}{\vert A_n(x)\vert^{1/p}}-\frac{\chi_{A_n(y)}}{\vert A_n(y)\vert^{1/p}}\right\Vert_p^p\\
 &\le\left\Vert\frac{\chi_{A_n(x)}}{\vert A_n(x)\vert}-\frac{\chi_{A_n(y)}}{\vert A_n(y)\vert}\right\Vert_1\\
 \end{align*}
 \begin{align*}
 & \le\frac{1}{\vert A_n(x)\vert}\vert A_n(x)\backslash A_n(y)\vert+\left\vert\frac{1}{\vert A_n(x)\vert}-\frac{1}{\vert A_n(y)\vert}\right\vert\vert A_n(x)\cap A_n(y)\vert+\frac{1}{\vert A_n(y)\vert}\vert A_n(y)\backslash A_n(x)\vert\\
 & \le \frac{\vert A_n(x)\Delta A_n(y)\vert\max\{\vert A_n(x)\vert,\vert A_n(y)\vert\}+\vert A_n(x)\cap A_n(y)\vert\cdot\vert\vert A_n(x)\vert-\vert A_n(y)\vert\vert}{\vert A_n(x)\vert\vert A_n(y)\vert}\\
 & \le \frac{\vert A_n(x)\Delta A_n(y)\vert[\max\{\vert A_n(x)\vert,\vert A_n(y)\vert\}+\vert A_n(x)\cap A_n(y)\vert]}{\vert A_n(x)\vert\vert A_n(y)\vert}\\
  & \le \frac{\vert A_n(x)\Delta A_n(y)\vert\cdot 2\max\{\vert A_n(x)\vert,\vert A_n(y)\vert\}}{\vert A_n(x)\vert\vert A_n(y)\vert}\\
  & \le \frac{2\vert A_n(x)\Delta A_n(y)\vert}{\min\{\vert A_n(x)\vert,\vert A_n(y)\vert\}}\\
  & \le \frac{2\vert A_n(x)\Delta A_n(y)\vert}{\vert A_n(x)\cap A_n(y)\vert}\le 2\epsilon_n\textrm{ whenever }d(x,y)\le r_n\\
  \end{align*}  
\end{proof}

It is by now routine to estimate the compression function in terms of the generalized inverse of the radial dilation function.

\begin{proposition}\label{Aembedding}
Let $p\ge 1$. If a metric measured space $(\MM,d)$ admits an adapted fast $A$-collection then there is a coarse embedding  $\phi$ from $\MM$ into $\ell_p(\MM,\R)$ such that
$$rad_{\A}^{-}(d(x,y))^{1/p}\lesssim_l \Vert \phi(x)-\phi(y)\Vert_p\lesssim_l d(x,y)^{1/p}$$
\end{proposition}

An typical application of proposition \ref{Aembedding} is presented below.
\begin{example} Let $(\mathcal{T},d)$ be a connected metric tree with an infinite geodesic equipped with the Lebesgue measure. Choose a root $t_0\in \mathcal{T}$ and denote $\omega_{t_0}$ the infinite geodesic ray starting at $t_0$. For any point $t\in\mathcal{T}$ there exists a unique geodesic ray $\omega_t$ such that $\omega_t\cap\omega_{t_0}$ is infinite. Define $A_n(t)$ to be the geodesic segment of length $r_n\epsilon_n^{-1}$ on the ray $\omega_t$. \\Then $A_n(t)\subset B(t,r_n\epsilon_n^{-1})$ and $d(x,y)\le r_n$ implies $\vert A_n(x)\Delta A_n(y)\vert\le 2r_n$ while $\vert A_n(x)\cap A_n(y)\vert\ge r_n\epsilon_n^{-1}$. Choosing $r_n=n$, $\epsilon_n=\frac{1}{n\log^2(n)}$ one has $rad_{\A}(n)\le r_n\epsilon_n^{-1}=n^{2}\log(n)^{2}$ hence $h_{(2,2)}\ll rad_{\A}^-$. From proposition \ref{Aembedding} it follows that there is a coarse embedding  $\phi$ from $\mathcal{T}$ into $\ell_p(\mathcal{T},\R)$ so that
$$h_{(2,2)}(d_{\mathcal{T}}(x,y))^{1/p}\lesssim_l \Vert \phi(x)-\phi(y)\Vert_p\lesssim_l d_{\mathcal{T}}(x,y)^{1/p}.$$
\end{example}

\begin{remark}
When  $t$ is large, $t^s\le h_{(2,2)}(t)\le t$ for every $0<s<1/2$ and at the embedding provide a lower bound $\frac{1}{4}$ for the Hilbert compression. The best value that can be achieved from our embedding is $\frac{1}{2}$ but the Hilbert compression of a tree is $1$ (see for instance \cite{CampbellNiblo2005}). It is not surprising that our embedding cannot achieve the best compression since our approach can be applied in a wide setting and requires mild geometric properties for the domain space. Nevertheless it would be interesting to know wether or not is it possible to construct an A-collection inducing an embedding with compression $\frac{1}{2}$.
\end{remark}

\subsection{Metric locally compact amenable groups}\ \\

\smallskip

Every locally compact group admits a Haar measure denoted by the symbol $\vert\cdot\vert$ in the sequel. By a \textit{metric group} one means a group endowed with a left-invariant metric. If the topology induced by the metric is locally compact the group is said to be a \textit{ metric locally compact group}. A metric locally compact group $(G,d_G)$ is said to admit an \textit{$(\varepsilon,R)$-F\o lner set} $F$ if there exists a compact set $F\subset G$ such that $\frac{\vert F\Delta gF\vert}{\vert F\vert}\le \varepsilon$ for every $g\in G$ satisfying $d_G(e,g)\le R$. Let $\epsilon=(\epsilon_n)_{n\ge 1}$ and $r=(r_n)_{n\ge 1}$. An $(\epsilon,r)$-F\o lner sequence is a sequence of compact sets $(F_n)_{n\ge 1}$ such that for every $n\ge 1$ $F_n$ is an $(\epsilon_n,r_n)$-F\o lner set. There is a close relation between F\o lner sequences and $A$-collections.

\begin{lemma}\label{AFolner}Let $G$ be a metric group and $\epsilon_n<1$ for all $n\ge 1$.

\begin{enumerate}

\item If $(F_n)_{n\ge 1}$ is an $(\epsilon,r)$-F\o lner sequence then the collection $(gF_n)_{(g,n)\in G\times\N}$ is an $(\epsilon',r)$-A collection with $\epsilon'_n=\frac{\epsilon_n}{1-\epsilon_n}$.\\
\item If a collection of the form $(gF_n)_{(g,n)\in G\times\N}$ is an $(\epsilon,r)$-A collection then $(F_n)_{n\ge 1}$ is an $(\epsilon,r)$-F\o lner sequence.
\end{enumerate}
\end{lemma}

\begin{proof}
(1) Remark that 
\begin{align*}
\frac{\vert gF_n\Delta hF_n\vert}{\vert gF_n\cap hF_n\vert} &=\frac{\vert F_n\Delta g^{-1}hF_n\vert}{\vert F_n\cap g^{-1}hF_n\vert}=\frac{\vert F_n\Delta g^{-1}hF_n\vert}{\vert F_n\vert}\frac{\vert F_n\vert}{\vert F_n\cap g^{-1}hF_n\vert}\\
\end{align*}

and 

\begin{align*}
1\ge \frac{\vert F_n\cap g^{-1}hF_n\vert}{\vert F_n\vert} &=\frac{\vert F_n\cup g^{-1}hF_n\vert}{\vert F_n\vert}-\frac{\vert F_n\Delta g^{-1}hF_n\vert}{\vert F_n\vert}\ge 1-\frac{\vert F_n\Delta g^{-1}hF_n\vert}{\vert F_n\vert}.\\
\end{align*}

Since $\frac{\vert F_n\Delta g^{-1}hF_n\vert}{\vert F_n\vert}\le \epsilon_n$ when $d_G(e,g^{-1}h)=d_G(g,h)\le r_n$ one obtains $\frac{\vert gF_n\Delta hF_n\vert}{\vert gF_n\cap hF_n\vert}\le \epsilon_n\ds\frac{1}{1-\epsilon_n}$ as long as $d_G(g,h)\le r_n$.

Morover the F\o lner sets are compact hence there exists a number $S_n$ such that $F_n$ is included in the ball centered at $e_G$ and of radius $S_n$. It implies that $gF_n\subset gB(e_G,S_n)=B(g,S_n)$ and condition $(1)$ of property A is satisfied.

\bigskip

(2)\begin{align*}
\frac{\vert F_n\Delta gF_n\vert}{\vert F_n\vert}=\frac{\vert F_n\Delta gF_n\vert}{\vert F_n\cap gF_n\vert}\frac{\vert F_n\cap gF_n\vert}{\vert F_n\vert}\le \frac{\vert F_n\Delta gF_n\vert}{\vert F_n\cap gF_n\vert}\le \epsilon_n\\
\end{align*}
if $d_G(e_G,g)\le r_n$
\end{proof}

An \textit{adapted fast $(\epsilon,r)$-F\o lner sequence} is an $(\epsilon,r)$-F\o lner sequence inducing an adapted fast A-collection. By left-invariance of the metric the F\o lner sets can always be assumed to contain the unital element $e_G$ of the group $G$. 
The \textit{circumradius}, denoted $rad_c$, of a set containing the unital element is the smallest radius of a circle containing it. For a F\o lner sequence the radial dilation function of the induced $A$-collection is simply given by $rad_c(F_n)$. One of the numerous characterization of amenability in \cite{Pier1984} is:\\ 
A locally compact group $G$ is amenable if and only if for every $\varepsilon>0$ and every compact subset $K$ of $G$ there exists a compact subset $F$ in $G$ such that $0<\vert F\vert<\infty$ and $\frac{\vert F\Delta gF\vert}{\vert F\vert}<\varepsilon$ for every $g\in K$.\\
Clearly, every metric locally compact group $(G,d_G)$ is amenable if and only if it admits a fast adapted F\o lner sequence. Suppose that $(G,d_G)$ is amenable then one can associate a function 
\begin{align*}
R\o l_G(n):=\min\{ & rad_c(F); \textrm{ there is a compact subset F of G such that} \\
 & \frac{\vert F\Delta gF\vert}{\vert F\vert}\le \frac{1}{n\log(n)^2} \textrm{ whenever } d_G(e,g)\le n\}
\end{align*}
This function shall be called the \textit{radial F\o lner function} of $G$.

\begin{proposition}\label{AGembedding}
Let $p\ge 1$. A metric locally compact amenable group $G$ admits a coarse embedding  $\phi$ from $G$ into $\ell_p(G,\R)$ such that
$$R\o l_G^{-}(d_G(x,y))^{1/p}\lesssim_l \Vert \phi(x)-\phi(y)\Vert_p\lesssim_l d_G(x,y)^{1/p}$$
\end{proposition}

\begin{proof}
Without loss of generality one can always assume that the adapted fast F\o lner sequence is an $(\epsilon,r)$-F\o lner sequence such that $\epsilon_n\le 1/2$. From Lemma \ref{AFolner} the collection $(gF_n)_{(g,n)\in G\times\N}$ is a $(2\epsilon,r)$-A collection and one can apply Proposition \ref{Aembedding}
\end{proof}
In particular,

\begin{corollary}\label{Austinquestion}
Let $p\ge 1$. A metric locally compact amenable group $(G,d_G)$ with radial F\o lner function growing at most exponentially admits a coarse embedding  $\phi$ from $G$ into $\ell_p(G,\R)$ such that
$$\log(d_G(x,y))^{1/p}\lesssim_l \Vert \phi(x)-\phi(y)\Vert_p\lesssim_l d_G(x,y)^{1/p}$$ 
\end{corollary}

\begin{remark}
Corollary \ref{Austinquestion} has to be compared to Tim Austin's question
\cite{Austin2011} whether or not every \textit{finitely generated amenable} group admits an embedding into $L_p$ ($p\ge 1$) of the form
$$\log (d(x,y))\lesssim \Vert \phi(x)-\phi(y)\Vert_p\lesssim d(x,y) \textrm{ for all } x,y\in G.$$
It turns out that Austin's question has been answered negatively by Olshanskii and Osin \cite{OlshanskiiOsin}.
Corollary \ref{Austinquestion} tells us that at least for groups with F\o lner sequences not expanding too fast it is true (up to the exponent $\frac{1}{p}$). 
\end{remark}
\begin{remark} The \textit{volumic F\o lner function} for locally compact groups extending the classical F\o lner function $F\o l$ originally introduced for finitely generated groups \cite{Vershik1982} could be defined as follows: 
\begin{align*}
V\o l_G(n):=\min\{ & \vert F\vert; \textrm{ there is a compact subset F of G such that} \\
 & \frac{\vert F\Delta gF\vert}{\vert F\vert}\le \frac{1}{n} \textrm{ whenever } d_G(e,g)\le 1\}
\end{align*} 
Clearly for finitely generated groups $V\o l\sim F\o l$.
A simple volume estimate for a finitely generated amenable group provides the lower bound $R\o l(n)\ge \log(F\o l(n))$. It is worth mentioning that Erschler proved \cite{Erschler2003} that the group $G=((\dots((\Z\wr\Z)\wr\Z)\dots\wr\Z)\dots))$ (k times) has F\o lner function growing at least as fast as $n^{n^{n^{\dots^n}}}$. This group is an example of an amenable group with a radial F\o lner function which is super-exponential and it forces the compression function, from the coarse embedding constructed using geometric properties of F\o lner sets, to grow very slowly. Moreover it is possible to construct for an arbitrary monotone increasing $\Z+$-valued function $\kappa$, a finitely generated amenable group such that $\kappa\ll F\o l$ (see \cite{Gromov2008}).
\end{remark}

\subsection{Examples}\ \\

Two basic examples are given below to illustrate the utilization of proposition
\ref{AGembedding}. It seems clear that a thorough investigation of radial F\o lner functions for amenable groups deserves full attention but this interesting topic shall be differ to a subsequent article. Recall that the inverse of the function $t\mapsto t^a\log^b(t)$ is denoted $h_{(a,b)}$ and the asymptotics of this function are $t^{1/(a+c)}\le h_{(a,b)}(t)\le t^{1/a}$ for every $c>0$.

\begin{example} Let $k\ge 1$ and $G=\R^k$ equipped with the distance induced by the Euclidean norm.

\medskip

\noindent Consider $F_n=B(0,r_n\epsilon_n^{-1/k})$. For elements $g$ of length $\Vert g\Vert$ less than $r_n$ one has $\vert F_n\Delta gF_n\vert\le (2r_n)^k$. Hence $\frac{\vert F_n\Delta gF_n\vert}{\vert F_n\vert}\le\frac{(2r_n)^k}{(2r_n\epsilon_n^{-1/k})^k}=\epsilon_n$ whenever $\Vert g\Vert\le r_n$. Choosing $r_n=n$, $\epsilon_n=\frac{1}{n\log^2(n)}$ gives $R\o l(n)\le r_n\epsilon_n^{-1/k}=n^{(k+1)/k}\log(n)^{2/k}$ and a coarse embedding  $\phi$ from $G$ into $\ell_p(G,\R)$ such that
$$h_{(\frac{k+1}{k},\frac{2}{k})}(\Vert x-y\Vert)^{1/p}\lesssim_l \Vert \phi(x)-\phi(y)\Vert_p\lesssim_l \Vert x-y\Vert^{1/p}.$$

\end{example}

\begin{example} Let $G=\mathbb{H}_{\R}$ be the real Heisenberg group endowed with its left-invariant Carnot-Carath\'{e}odory metric $d_{cc}$.

\medskip

\noindent The Heisenberg group can be described as the group of $3\times 3$ matrices of the form $$g(x,y,z)=\left(\begin{array}{ccc}
1 & x & z\\
0 & 1 & y\\
0 & 0 & 1\\
\end{array}\right)$$
where $x,y,z\in \R$. The following geometric properties of the Heisenberg group can be found in \cite{Roe2003}. Any left-invariant metric is quasi-isometrically equivalent to the left-invariant metric given by $d(g(x,y,z),g(0,0,0))=\vert x\vert+\vert y\vert +\vert z\vert^{1/2}$.

\medskip

\noindent Consider $F_n=B_{d}(0,\frac{1}{\epsilon_n})$. For every $r_n>0$ and every $g\in \mathbb{H}_{\R}$ such that $d(e,g)\le r_n$ one has $\frac{\vert F_n\Delta gF_n\vert}{\vert F_n\vert}\lesssim\frac{\epsilon_n^{-3}}{\epsilon_n^{-4}}=\epsilon_n$. Taking $r_n=n$, $\epsilon_n=\frac{1}{n\log^2(n)}$ provides $R\o l(n)\le \frac{1}{\epsilon_n}=n\log^2(n)$ and a coarse embedding  $\phi$ from $G$ into $\ell_p(G,\R)$ such that
$$h_{(1,2)}(d_{cc}(x,y))^{1/p}\lesssim_l \Vert \phi(x)-\phi(y)\Vert_p\lesssim_l d_{cc}(x,y)^{1/p}.$$
\end{example}

\smallskip

\begin{remark}
According to the asymptotics of $h_{(1,2)}$, $\alpha_{\ell_p}(\mathbb{H}_\R)\ge 1/p$ and in particular $\alpha_{\ell_1}(\mathbb{H}_\R)=1$. It is worth mentioning that $\alpha_{\ell_p}(\mathbb{H}_\R)=1$ for any $p\ge 1$ ($\mathbb{H}_\R$ is doubling) and $\mathbb{H}_\R$ does not bi-Lipschitzly (actually quasi-isometrically) embed into $L_1$ or any superreflexive Banach space (see \cite{CheegerKleinerNaor2011} and \cite{AustinNaorTessera2010}). Using algebraic techniques and sequences of controlled F\o lner pairs Tessera \cite{Tessera2011} was able to prove better estimates for the compression modulus for certain groups and in particular the Heisenberg group.
\end{remark}

\section{Conclusion}
In the next page, we gather, for further reference, two tables where are collected the results, either from the current paper, from articles quoted in the core of the text or classical results, that are related to the Lebesgue compression of the classical Lebesgue sequence and function spaces. The attentive reader will easily spot the natural questions that are still open. We decided to isolate in the first table the compression estimates involving only Banach spaces as domain and target spaces, the Banach space setting being the most intriguing and relevant one for the geometric group theorists community. In the second table we include results featuring a mix of quasi-Banach spaces and Banach spaces which are of interest to nonlinear Banach space geometers.

\vskip2cm

\noindent{\it Acknowledgments.} I would like to thank the participants and organizers of the ``Groups and Dynamics Seminar'' at Texas A\& M University, in particular Slava Grigorschuk, for their useful comments and remarks that helped improve significantly the group theoretic section of the article. I am very grateful to Tim Austin for drawing to my attention the paper from Olshanskii and Osin. Finally I would like to express my deepest gratitude to Bill Johnson for his constant support and for simply sharing with me numerous ``mathematical moments''.
\newpage{}

\subsection{The Banach space setting}\ \\
\begin{tiny}
\begin{table}[ht]
\begin{center}
\begin{tabular}{|c|c|c|}
\hline
 & &\\
$\MM$ &  $\ell_q$-compression of $\MM$ & $L_q$-compression of $\MM$\\
 & &\\
 \hline
 & & \\
 $\ell_p$&  $\begin{array}{l}
                          \alpha_{\ell_q}(\ell_p)=0 \textrm{ if } 2\le q<p \textrm{ or } 1\le q\le 2< p\\
                           \\
                           \alpha_{\ell_q}(\ell_2)=1 \textrm{ if } 1\le q\le 2  \\
                 
                          \\
                          \alpha_{\ell_q}(\ell_p)= \frac{p}{q} \textrm{ if } 1\le p<q<\infty  \\
                         
                           \\
                            \frac{p}{2}\le \alpha_{\ell_q}(\ell_p)\le 1 \textrm{ if } 1\le q<p<2
                                                     \end{array}$ & 
             $\begin{array}{l}
             \alpha_{L_q}(\ell_p)=0 \textrm{ if } 2\le q<p \textrm{ or } 1\le q\le 2< p\\
                           \\
                           \alpha_{L_q}(\ell_2)=1 \textrm{ if } 1\le q\le \infty  \\
                           \\
                           \alpha_{L_q}(\ell_p)=1 \textrm{ if } 1\le q\le p\le 2\\
                         
                          \\
                          \alpha_{L_q}(\ell_p)= \frac{p}{\min\{q,2\}} \textrm{ if } 1\le p\le 2 \textrm{ and } q\ge p  \\
                         \\
                            \frac{p}{q}\le \alpha_{L_q}(\ell_p)\le 1 \textrm{ if } 2< p\le q<\infty  \\
                            \\

                                                     \end{array}$\\     
\hline
 & & \\
 $L_p$&  $\begin{array}{l} \alpha_{\ell_q}(L_2)=\alpha_{\ell_q}(\ell_2)\\
                          \\
                           \alpha_{\ell_q}(L_p)=0 \textrm{ if } 2\le q<p \textrm{ or } 1\le q\le 2< p\\
\\
                            \frac{p}{2}\le \alpha_{\ell_q}(L_p)\le 1 \textrm{ if } 1\le p,q\le2\\
                            \\
                            0\le \alpha_{\ell_q}(L_q)\le \frac{2}{q} \textrm{ if } q>2\\
                            \\

                                                                                \end{array}$ & 
             $\begin{array}{l}
             \alpha_{L_q}(L_2)=\alpha_{L_q}(\ell_2)\\
                          \\
             \alpha_{L_q}(L_p)=0 \textrm{ if } 2\le q<p \textrm{ or } 1\le q\le 2< p\\
                                   \\
                            \alpha_{L_q}(L_p)=1 \textrm{ if } 1\le q\le p\le 2 \\
                            \\
                            \alpha_{L_q}(L_p)=\frac{p}{\min\{q,2\}} \textrm{ if } 1\le p\le 2 \textrm{ and } q\ge p\\
                            \\
                                                        \frac{p}{q}\le \alpha_{L_q}(L_p)\le 1 \textrm{ if } 2< p\le q<\infty  \\
                            \\

                                                     \end{array}$\\
 & & \\      
\hline
\end{tabular}
\end{center}
\end{table}
\end{tiny}

\subsection{The mixed quasi-Banach space and Banach space setting}\ \\ 

\begin{tiny}
\begin{table}[ht]
\begin{center}
\begin{tabular}{|c|c|c|}
\hline
 & &\\
$\MM$ &  $\ell_q$-compression of $\MM$ & $L_q$-compression of $\MM$\\
 & &\\
 \hline
 & & \\
 $\ell_p$&  $\begin{array}{l}
                          \alpha_{\ell_q}(\ell_p)=0 \textrm{ if } 0<q<1 \textrm{ and }p>2\\
                                                     \\
                          q^2\le \alpha_{\ell_q}(\ell_2)\le 1 \textrm{ if } 0< q<1  \\
                          \\
                     \frac{pq^2}{2}\le \alpha_{\ell_q}(\ell_p)\le 1 \textrm{ if } 0< q\le 1< p< 2  \\
                           \\
                           \alpha_{\ell_q}(\ell_p)=1 \textrm{ if } 0<p<q\le 1  \\
                                                  \\
                                \frac{q^2}{2}\le \alpha_{\ell_q}(\ell_p)\le 1 \textrm{ if } 0< q<p\le 1  \\
                          \\
                          \alpha_{\ell_q}(\ell_p)= \frac{1}{q} \textrm{ if } 0< p\le 1<q<\infty\\
                          \\

                                                     \end{array}$ & 
             $\begin{array}{l}
             \alpha_{L_q}(\ell_p)=0 \textrm{ if }0<q<1 \textrm{ and }p>2\\
                           \\
                           \alpha_{L_q}(\ell_2)=1 \textrm{ if } 0< q<1  \\
                           \\
                 \alpha_{L_q}(\ell_p)=1 \textrm{ if } 0< q\le 1< p< 2  \\
                           \\

                             \alpha_{L_q}(\ell_p)=1 \textrm{ if } 0< p,q\le 1  \\
                                                   \\
                         \alpha_{L_q}(\ell_p)=\frac{1}{\min\{q,2\}} \textrm{ if } 0< p\le 1<q<\infty\\
                                                   \\

                                                     \end{array}$\\     
\hline
 & & \\
 $L_p$&  $\begin{array}{l}
                           \alpha_{\ell_q}(L_p)=0 \textrm{ if } 0<q<1 \textrm{ and }p>2\\
\\
                           \frac{pq^2}{2}\le \alpha_{\ell_q}(L_p)\le 1 \textrm{ if } 0< q\le 1< p< 2  \\
                           \\  
                            \frac{q^2}{2}\le \alpha_{\ell_q}(L_p)\le 1 \textrm{ if } 0< p,q\le 1  \\
                          \\
                                                    \frac{1}{\max\{q,2\}}\le \alpha_{\ell_q}(L_p)\le \frac{1}{q} \textrm{ if } 0< p\le 1<q<\infty\\
                          \\
                                                                                                        \end{array}$ & 
             $\begin{array}{l}
             \alpha_{L_q}(L_p)=0 \textrm{ if } 0<q<1 \textrm{ and }p>2\\
                           \\
                       \alpha_{L_q}(L_p)=1 \textrm{ if } 0< q\le 1< p< 2  \\
                           \\  
                            \alpha_{L_q}(L_p)=1 \textrm{ if } 0< p,q\le 1  \\
                                                   \\
                      \alpha_{L_q}(L_p)=\frac{1}{\min\{q,2\}} \textrm{ if } 0< p\le 1<q<\infty\\
                                                   \\
                                                     \end{array}$\\
 & & \\      
\hline
\end{tabular}
\end{center}
\end{table}
\end{tiny}
\newpage{}

\section{Appendix}
\paragraph{\textbf{Mazur maps}} Let $0<p,q<\infty$ and define
$$\begin{array}{rccl}
M_{p,q}\colon& \ell_p & \to & \ell_q\\
 & (x_n)_{n\ge 1}&\mapsto & (sgn(x_n)\vert x_n\vert^{p/q})_{n \ge 1}
\end{array}$$
If $\vert x\vert_{p}=1$ then $\vert M_{p,q}(x)\vert_{q}=1$. The Mazur map estimates are gathered in Proposition \ref{Mazur} below.
\begin{proposition}[Mazur map estimates]\label{Mazur}\ \\
\begin{enumerate}
\item $1\le q<p<\infty$, $$ \Vert x-y\Vert_p^{p/q}\lesssim \Vert \M_{p,q}(x)-M_{p,q}(y)\Vert_q\lesssim\Vert x-y\Vert_p $$
\item $0< q\le 1\le p<\infty$, $$ \Vert x-y\Vert_p^{p}\lesssim d_q(M_{p,q}(x),M_{p,q}(y))\lesssim\Vert x-y\Vert_p^{q} $$
\item $0< q<p\le 1$, $$ d_p(x,y)\lesssim d_q(\M_{p,q}(x),M_{p,q}(y))\lesssim d_p(x,y)^{q/p}$$
\end{enumerate}
\end{proposition}

\bigskip

\begin{remark}
In the case $p<q$ the inequalities are reversed.\end{remark}

The proof of the next lemma is left to the reader.
\begin{lemma}
For $\alpha\ge1$ we have 

\begin{enumerate}
\item $\vert sgn(a)\vert a\vert^\alpha-sgn(b)\vert b\vert^\alpha\vert\ge c_\alpha\vert a-b\vert^\alpha$\\
\item $\vert sgn(a)\vert a\vert^\alpha-sgn(b)\vert b\vert^\alpha\vert\le\alpha\vert a-b\vert\max\{\vert a\vert,\vert b\vert\}^{\alpha-1}$
\end{enumerate}
\end{lemma}

Proposition \ref{Mazur} is a direct consequence of Lemma \ref{inequalities}.
\begin{lemma}\label{inequalities}Let $0<q<p<\infty$\ \\
If $\vert x\vert_{p}=1$ and $\vert y\vert_{p}=1$ then $$\sum_{n=1}^{\infty}\vert x_n-y_n\vert^{p}\lesssim\sum_{n=1}^\infty\vert M_{p,q}(x)(n)-M_{p,q}(y)(n)\vert^q\lesssim \left(\sum_{n=1}^{\infty}\vert x_n-y_n\vert^{p}\right)^{q/p}$$ 
\end{lemma}

\bigskip

\begin{proof}
\begin{align*}
\sum_{n=1}^\infty\vert M_{p,q}(x)(n)-M_{p,q}(y)(n)\vert^q&=\sum_{n=1}^{\infty}\vert sgn(x_n)\vert x_n\vert^{p/q}-sgn(y_n)\vert y_n\vert^{p/q}\vert^q\\
&\ge c_{p/q}^q\sum_{n=1}^{\infty}\vert x_n-y_n\vert^{p}
\end{align*}
and,
\begin{align*}
\sum_{n=1}^\infty\vert M_{p,q}(x)(n)-M_{p,q}(y)(n)\vert^q&\le\left(\frac{p}{q}\right)^q\sum_{n=1}^{\infty}\left(\vert x_n-y_n\vert^{q}\max\{\vert x_n\vert, \vert y_n\vert\}^{p-q}\right)
\end{align*}
\begin{align*}
\le\left(\frac{p}{q}\right)^q \left(\sum_{n=1}^{\infty}\vert x_n-y_n\vert^{p}\right)^{q/p}\left(\sum_{n=1}^\infty\max\{\vert x_n\vert, \vert y_n\vert\}^{p}\right)^{\frac{p-q}{p}}\\
\le\left(\frac{p}{q}\right)^q\cdot 2^{1-q/p}\left(\sum_{n=1}^{\infty}\vert x_n-y_n\vert^{p}\right)^{q/p}\\
\end{align*}
\end{proof}

\end{document}